\documentclass{amsart}
\usepackage{amsfonts,amssymb,amsmath,amsthm,amscd}
\usepackage{url}
\usepackage{enumerate}

\newtheorem{theorem}{Theorem}[section]
\newtheorem{lemma}[theorem]{Lemma}
\newtheorem{proposition}[theorem]{Proposition}
\newtheorem{corollary}[theorem]{Corollary}
\newtheorem{conjecture}[theorem]{Conjecture}

\theoremstyle{definition}

\newenvironment{definition}[1][Definition]{\begin{trivlist}
\item[\hskip \labelsep {\bfseries #1}]}{\end{trivlist}}
\newenvironment{remark}[1][Remark]{\begin{trivlist}
\item[\hskip \labelsep {\bfseries #1}]}{\end{trivlist}}

\DeclareFontEncoding{OT2}{}{}

\newcommand{\ilim}[1]{\displaystyle{\lim_{\genfrac{}{}{0pt}{}{\longleftarrow}{\scriptstyle #1}} }\;}

\newcommand{\ord}{{\operatorname{ord}}}
\newcommand{\Gal}{{\operatorname{Gal}}}

\newcommand{\Ram}{{\operatorname{Ram}}}
\newcommand{\Res}{{\operatorname{Res}}}
\newcommand{\rec}{{\operatorname{rec}}}

\newcommand{\nrd}{{\operatorname{nrd}}}
\newcommand{\tors}{{\operatorname{tors}}}

\newcommand{\Stab}{{\operatorname{Stab}}}

\newcommand{\End}{{\operatorname{End}}}

\newcommand{\Hom}{{\operatorname{Hom}}}
\newcommand{\Sel}{{\operatorname{Sel}}}

\newcommand{\GL}{{\operatorname{GL_2}}}

\newcommand{\PGL}{{\operatorname{PGL_2}}}
\newcommand{\M}{{\operatorname{M_2}}}
\newcommand{\JL}{{\operatorname{JL}}}

\newcommand{\UU}{\mathcal{U}}
\newcommand{\HH}{\mathcal{H}}

\author{Jeanine Van Order}
\address{Section de Math\'ematiques\\ Ecole Polytechnique F\'ed\'erale de Lausanne\\ Lausanne 1015, Switzerland}
\email{jeanine.vanorder@epfl.ch}
\thanks{The author acknowledges support from the Swiss National Science Foundation (FNS) grant 200021-125291.}

\begin{document}

\title[Quaternionic $p$-adic $L$-functions of Hilbert modular eigenforms]
{On the quaternionic $p$-adic $L$-functions associated to Hilbert modular eigenforms}

\begin{abstract}  We construct $p$-adic $L$-functions associated to 
cuspidal Hilbert modular eigenforms of parallel weight two in certain
dihedral or anticyclotomic extensions via the Jacquet-Langlands correspondence, 
generalizing works of Bertolini-Darmon, Vatsal and others. The construction
given here is adelic, which allows us to deduce a precise interpolation formula from
a Waldspurger type theorem, as well as a formula for the dihedral 
$\mu$-invariant. We also make a note of Howard's nonvanishing criterion 
for these $p$-adic $L$-functions, which can be used to reduce the associated Iwasawa main 
conjecture to a certain nontriviality criterion for families of $p$-adic $L$-functions.\end{abstract}

\maketitle
\tableofcontents

\section{Introduction}

Let $F$ be a totally real number field of degree $d$ over ${\bf{Q}}$. 
Fix a prime ideal $\mathfrak{p} \subset \mathcal{O}_F$ with underlying
rational prime $p$.  Fix an integral ideal $\mathfrak{N}_0 
\subset \mathcal{O}_F$ with $\ord_{\mathfrak{p}}(\mathfrak{N}_0)\leq 1$. Let 

\begin{align}\label{N}
\mathfrak{N} &= \begin{cases}
\mathfrak{N}_0 &\text{ if $\mathfrak{p} \mid \mathfrak{N}_0$}\\
\mathfrak{p}\mathfrak{N}_0 &\text{ if $\mathfrak{p} \nmid \mathfrak{N}_0$.}
\end{cases}
\end{align} Hence, $\ord_{\mathfrak{p}}(\mathfrak{N})=1$.
Fix a totally imaginary quadratic extension $K$ of $F$. Assume that
the relative discriminant $\mathfrak{D}_{K/F}$ of $K$ over $F$ is 
prime to $\mathfrak{N}/\mathfrak{p}$. The choice of $K$ then determines
uniquely a factorization \begin{align}\label{factorization}\mathfrak{N} 
= \mathfrak{p} \mathfrak{N}^{+} \mathfrak{N}^{-}\end{align} 
of $\mathfrak{N}$ in $\mathcal{O}_F$, with $v
\mid \mathfrak{N}^{+}$ if and only if $v$ splits in $K$, and $v \mid
\mathfrak{N}^{-}$ if and only if $v$ is inert in $K$. Let us assume
additionally that $\mathfrak{N}^{-}$ is the squarefree product
of a number of primes congruent to $d \mod 2$. 
Fix a Hilbert modular eigenform ${\bf{f}} \in \mathcal{S}_2(\mathfrak{N})$ of parallel weight 
two, level $\mathfrak{N}$, and trivial Nebentypus. Assume that ${\bf{f}}$ is either a
newform, or else arises from a newform of level $\mathfrak{N}/\mathfrak{p}$
via the process of $\mathfrak{p}$-stabilization. Our hypotheses on $\mathfrak{N}$ and 
$K$ imply that the global root number of the associated Rankin-Selberg $L$-function 
$L({\bf{f}} , K, s)$ at its central value $s =1$ is equal to $1$ (as opposed to $-1$). Moreover, 
the Jacquet-Langlands correspondence allows us to associated to ${\bf{f}} $ an eigenform 
on a totally definite quaternion algebra over $F$, which puts us in the setting of Waldspurger's 
theorem \cite{Wa}, as refined by Yuan-Zhang-Zhang in \cite{YZ^2}. Let us view ${\bf{f}} $ as a 
$p$-adic modular form via a fixed embedding $\overline{{\bf{Q}}} \rightarrow  
\overline{{\bf{Q}}}_p,$ writing $\mathcal{O}$ to denote the ring of integers of a 
finite extension of ${\bf{Q}}_p$ containing all of the Fourier coefficients of 
${\bf{f}}$ under this fixed embedding. Let us assume 
additionally that ${\bf{f}} $ is either $\mathfrak{p}$-ordinary, by which we mean that its eigenvalue 
at the Hecke operator $T_{\mathfrak{p}}$ is invertible in $\mathcal{O}$, or else that 
${\bf{f}} $ is $\mathfrak{p}$-supersingular, by which we main that its eigenvalue at the Hecke 
operator $T_{\mathfrak{p}}$ is zero. In the case where ${\bf{f}}$ is $\mathfrak{p}$-ordinary, 
let us write $\alpha_{\mathfrak{p}} = \alpha_{\mathfrak{p}}({\bf{f}} )$ to denote the unit root 
of the Hecke polynomial $X^2 - a_{\mathfrak{p}}({\bf{f}} )X + q.$ Here, $a_{\mathfrak{p}}({\bf{f}})$ 
denotes the eigenvalue of ${\bf{f}}$ at $T_{\mathfrak{p}}$, and $q$ denotes the cardinality of the 
residue field of $\mathfrak{p}$. Let $\delta = [F_{\mathfrak{p}}: {\bf{Q}}_p]$. We consider 
the behaviour of ${\bf{f}}$ in the dihedral ${\bf{Z}}_p^{\delta}$-extension 
$K_{\mathfrak{p}^{\infty}}$ of $K$ described by class field  theory. Writing the 
Galois group $\Gal(K_{\mathfrak{p}^{\infty}}/K) \cong {\bf{Z}}_p^{\delta}$
as \begin{align*} G_{\mathfrak{p}^{\infty}} = \ilim n G_{\mathfrak{p}^n}\end{align*} 
we consider the $\mathcal{O}$-Iwasawa algebra 
\begin{align*} \Lambda = \mathcal{O} [[G_{\mathfrak{p}^{\infty}}]] = \ilim n 
\mathcal{O}[G_{\mathfrak{p}^{n}}],\end{align*} whose elements can be viewed as $\mathcal{O}$-valued 
measures on $G_{\mathfrak{p}^{\infty}}$. The purpose of this note is to give a construction of 
elements $\mathcal{L}_{\mathfrak{p}}({\bf{f}}, K) \in \Lambda$ whose specializations to 
finite order characters $\rho$ of $G_{\mathfrak{p}^{\infty}}$ interpolate the central
values $L({\bf{f}}, \rho, 1)$ of the twisted Rankin-Selbert $L$-functions $L({\bf{f}}, \rho, s)$. To be 
more precise, if $\rho$ is a finite order character of $G_{\mathfrak{p}^{\infty}}$, and $\lambda$ an
element of $\Lambda$ with associated measure $d\lambda$, let 
\begin{align*} \rho \left( \lambda \right) = \int_{G_{\mathfrak{p}^{\infty}}} \rho(\sigma) 
\cdot d\lambda(\sigma)\end{align*} denote the specialization of $\lambda$ to
$\rho$. Let $\mathfrak{P}$ denote the maximal ideal of $\mathcal{O}$. We define the 
$\mu$-invariant $\mu = \mu(\lambda)$ associated to an element $\lambda \in \Lambda$ to be 
the largest exponent $c$ such that $\lambda \in \mathfrak{P}^c \Lambda$. Let $\pi = \pi_{\bf{f}}$ 
denote the cuspidal automorphic representation of $\GL$ over $F$ associated to ${\bf{f}}$, with 
\begin{align*} L(\pi, \operatorname{ad}, s) = \prod_v L(\pi_v, \operatorname{ad}, s)\end{align*} 
the $L$-series of the adjoint representation of $\pi$. Let \begin{align*}\zeta_F(s) 
= \prod_v \zeta_v(s)\end{align*} denote the Dedekind zeta function of $F$. Let \begin{align*}
L(\pi, \rho, s) = \prod_v L(\pi_v, \rho_v, s)\end{align*} denote the Rankin-Selberg $L$-function 
of $\pi$ times the theta series associated to $\rho$, with central value at $s=1/2$. 
Note that we have an equality of $L$-functions \begin{align*} L(\pi, \rho, s- \frac{1}{2}) 
&= \Gamma_{\bf{C}}(s)^{[K:{\bf{Q}}]}L({\bf{f}}, \rho, s), \end{align*} with 
\begin{align*} \Gamma_{\bf{C}}(s)= 2(2\pi)^{-s}\Gamma(s),\end{align*}
as explained for instance in \cite[0.5]{Ne}. Let $\pi = \JL(\pi')$ denote the 
Jacquet-Langlands correspondence of $\pi$. Choose a decomposable vector 
$\Phi = \otimes_v \Phi_v \in \pi'$. Let $\omega = \omega_{K/F}$ denote the quadratic
Hecke character associated to $K/F$, with decomposition $\omega = \otimes_v \omega_v $.
Following Yuan-Zhang-Zhang \cite{YZ^2}, we consider for any prime $v$ of $F$ the local 
linear functional defined by \begin{align*}\alpha(\Phi_v, \rho_v) = \frac{ L(\omega_v, 1) \cdot L(\pi_v,
\operatorname{ad}, 1) }{\zeta_v(2) \cdot L(\pi_v, \rho_v, 1/2)} \cdot 
\int_{K_{v}^{\times}/F_{v}^{\times}} \langle \pi_v'(t)\Phi_v,
\Phi_v\rangle_v \cdot \rho_v(t) d t. \end{align*} Here, $\langle \cdot, \cdot \rangle_v$ 
denotes a nontrivial hermitian form on the component $\pi_v'$ such that the product 
$\langle \cdot, \cdot \rangle = \prod_v \langle \cdot, \cdot \rangle_v $ is the Petersson
inner product, and $dt$ denotes the Tamagawa measure. We refer the reader to the 
discussion below for more precise definitions. We show the following result.

\begin{theorem} [Theorem \ref{basicint}, Corollary \ref{nontriviality}, and Theorem \ref{mu}] 
Fix an eigenform  ${\bf{f}} \in \mathcal{S}_2(\mathfrak{N})$ subject to the hypotheses above, 
with $\overline{{\bf{Q}}} \rightarrow \overline{{\bf{Q}}}_p$ a fixed embedding. Let $\rho$ be a 
finite order character of the Galois group $G_{\mathfrak{p}^{\infty}}$ that factors though 
$G_{\mathfrak{p}^{n}}$ for some positive integer $n$. Let $\vert \cdot \vert$ denote the complex 
absolute value on $\overline{\bf{Q}}_p,$ taken with respect to a fixed embedding 
$\overline{\bf{Q}}_p \rightarrow \overline{\bf{C}}$.

\begin{itemize}
\item[(i)] If ${\bf{f}}$ is $\mathfrak{p}$-ordinary, then there exists a nontrivial 
element $\mathcal{L}_{\mathfrak{p}}({\bf{f}}, K_{\mathfrak{p}^{\infty}}) \in \Lambda$ such that 
the following equality holds in $\overline{\bf{Q}}_p$:

 \begin{align*} \vert \rho\left(\mathcal{L}_{\mathfrak{p}}({\bf{f}}, K_{\mathfrak{p}^{\infty}})
\right) \vert &= \frac{ \alpha_{\mathfrak{p}}^{-2n} \cdot
\zeta_F(2)}{ 2 \cdot L(\pi, \operatorname{ad}, 1)} \\
&\times  \left[ L(\pi, \rho, 1/2) \cdot L(\pi, \rho^{-1}, 1/2)
\cdot \prod_{v \nmid \infty} \alpha(\Phi_v, \rho_v) \cdot
\alpha(\Phi_v, \rho_v^{-1}) \right]^{\frac{1}{2}}. \end{align*}

\item[(ii)] If ${\bf{f}}$ is $\mathfrak{p}$-supersingular, then there 
exist two nontrivial elements $\mathcal{L}_{\mathfrak{p}}({\bf{f}}, K_{\mathfrak{p}^{\infty}})^{\pm} 
\in \Lambda$ such that the following equalities hold in $\overline{\bf{Q}}_p$:

\begin{align*} \vert \rho\left(\mathcal{L}_{\mathfrak{p}}({\bf{f}}, K_{\mathfrak{p}^{\infty}})^{\pm} 
\right) \vert &= \frac{\zeta_F(2)}{ 2 \cdot L(\pi, \operatorname{ad}, 1)} \\
&\times  \left[ L(\pi, \rho, 1/2) \cdot L(\pi, \rho^{-1}, 1/2)
\cdot \prod_{v \nmid \infty} \alpha(\Phi_v, \rho_v) \cdot
\alpha(\Phi_v, \rho_v^{-1}) \right]^{\frac{1}{2}}. \end{align*}
\end{itemize} 

\item[(iii)] We have that $\mu(\mathcal{L}_{\mathfrak{p}}(
{\bf{f}}, K_{\mathfrak{p}^{\infty}}) ) = \mu(\mathcal{L}_{\mathfrak{p}}(
{\bf{f}}, K_{\mathfrak{p}^{\infty}})^{\pm} ) = 2 \nu,$ where $\nu = \nu_{\Phi}$ 
denotes the largest integer such that  $\Phi$ is congruent to a constant 
$\mod \mathfrak{P}^{\nu}.$ \end{theorem}

At this point, some remarks are in order. First of all, we note that this construction 
is the generalization to totally real fields of that given by Bertolini-Darmon 
\cite[$\S$ 1.2]{BD} and Darmon-Iovita \cite{DI}, building on the seminal work of 
Bertolini-Darmon \cite{BD2}. It is sketched by Longo in \cite{Lo2}, using the language 
of Gross points. The novelty of the construction given here is that we work adelically 
rather than $p$-adically, which allows us to use Waldspurger's theorem directly to deduce 
the interpolation property. This also allows us to give a simpler proof of the $\mu$-invariant 
formula than that given by Vatsal in \cite{Va2}. Finally, this construction allows us to reduce the 
associated Iwasawa main conjecture to a nonvanishing criterion for these $p$-adic $L$-functions 
via the theorem of Howard \cite[Theorem 3.2.3(c)]{Ho} (cf. also \cite[Theorem 1.3]{VO2}), as we 
explain below. In particular, Howard's criterion (Conjecture \ref{HC}) has the following applications 
to Iwasawa main conjectures. Let $\Sel_{p^{\infty}}({\bf{f}}, K_{\mathfrak{p}^{\infty}})$ denote the 
$p^{\infty}$-Selmer group associated to ${\bf{f}}$ in the tower $K_{\mathfrak{p}^{\infty}}/K$. 
We refer the reader to \cite{Ho0}, \cite{Lo2} or \cite{VO2} for definitions. Let \begin{align*} 
X({\bf{f}}, K_{\mathfrak{p}^{\infty}}) = \operatorname{Hom}\left( \Sel_{p^{\infty}}({\bf{f}}, K_{\mathfrak{p}^{\infty}}), 
{\bf{Q}}_p/{\bf{Z}}_p \right)\end{align*} denote the Pontryagin dual of 
$\Sel_{p^{\infty}}({\bf{f}}, K_{\mathfrak{p}^{\infty}})$, which has the structure of a 
compact $\Lambda$-module. If $X({\bf{f}}, K_{\mathfrak{p}^{\infty}})$ is $\Lambda$-torsion, 
then the structure theory of $\Lambda$-modules in \cite[$\S$ 4.5]{BB} assigns to 
$X({\bf{f}}, K_{\mathfrak{p}^{\infty}}) $ a $\Lambda$-characteristic power series 
\begin{align*} \operatorname{char}_{\Lambda}\left( X({\bf{f}}, K_{\mathfrak{p}^{\infty}})  
\right) \in \Lambda.\end{align*} The associated dihedral main conjecture of Iwasawa 
theory is then given by

\begin{conjecture}[Iwasawa main conjecture]\label{mainconjecture} Let  
${\bf{f}} \in \mathcal{S}_2(\mathfrak{N})$ be a Hilbert modular eigenform as defined above, 
such that the global root number of the complex central value $L({\bf{f}},K, 1) $ is $+1$. Then, 
the dual Selmer group $X({\bf{f}} , K_{\mathfrak{p}^{\infty}}) $ is $\Lambda$-torsion, and there 
is an equality of ideals \begin{align}\label{fullequality} \left(  \mathcal{L}_{\mathfrak{p}}({\bf{f}}, 
K_{\mathfrak{p}^{\infty}}) \right) = \left( \operatorname{char}_{\Lambda}\left( X({\bf{f}} , 
K_{\mathfrak{p}^{\infty}})  \right) \right) \text{ in } \Lambda. \end{align} \end{conjecture} 
We may now deduce the following result towards this conjecture.

\begin{theorem}\label{IMC} Suppose that $F={\bf{Q}}.$ Let  ${\bf{f}} \in 
\mathcal{S}_2(\mathfrak{N})$ be a Hilbert modular eigenform as defined above. 
Assume that the residual Galois representation associated to ${\bf{f}}$ is surjective, 
as well as ramified at each prime $\mathfrak{q}\mid \mathfrak{N}^{-}$ such that 
$\mathfrak{q} \equiv \pm 1 \mod \mathfrak{p}.$  Then, the dual Selmer group 
$X({\bf{f}}, K_{\mathfrak{p}^{\infty}}) $ is $\Lambda$-torsion, and there is an inclusion 
of ideals \begin{align*} \left(\mathcal{L}_{\mathfrak{p}}({\bf{f}} , 
K_{\mathfrak{p}^{\infty}}) \right) \subseteq  
\left( \operatorname{char}_{\Lambda}\left( X({\bf{f}} , K_{\mathfrak{p}^{\infty}})  \right) 
\right) \text{ in } \Lambda. \end{align*} Moreover, if Conjecture \ref{HC} below holds, 
then there is an equality of ideals 
\begin{align*} \left(  \mathcal{L}_{\mathfrak{p}}({\bf{f}} , K_{\mathfrak{p}^{\infty}}) \right) =
 \left( \operatorname{char}_{\Lambda}\left( X({\bf{f}} , K_{\mathfrak{p}^{\infty}})  \right)  
\right) \text{ in } \Lambda. \end{align*}\end{theorem}

\begin{proof} The result follows from the refinement of the Euler system argument of 
Bertolini-Darmon \cite{BD} given by Pollack-Weston \cite{PW}, which satisfies the hypotheses 
of Howard \cite[Theorem 3.2.3]{Ho0}, since it removes the $p$-isolatedness condition from 
the work of \cite{BD}. $\Box$ \end{proof}  We can state the criterion of Conjecture \ref{HC} in 
the following more explicit way. Let us now assume for simplicity that $\mathfrak{N}$ is prime
to the relative discriminant of $K$ over $F$. Fix a positive integer $k$. Let us define a set of 
admissible primes $\mathfrak{L}_k$ of $F$, all of which are inert in $K$, with the condition 
that for any ideal $\mathfrak{n}$ in the set $\mathfrak{S}_k$ of squarefree products of primes 
in $\mathfrak{L}_k$, there exists a nontrivial eigenform ${\bf{f}}^{(\mathfrak{n})}$ of level 
$\mathfrak{nN}$ such that we have the following congruence on Hecke eigenvalues:
\begin{align*} {\bf{f}}^{(\mathfrak{n})} \equiv {\bf{f}} \mod \mathfrak{P}^k .\end{align*} Let
$\mathfrak{L}_{k}^{+} \subset \mathfrak{L}_{k}$ denote the subset of primes $v$ for which 
$\omega(v\mathfrak{N}) = -1$, equivalently for which the root number of the Rankin-Selberg 
$L$-function $L({\bf{f}}^{(\mathfrak{n})}, K, s)$ equals $1$. Let $\mathfrak{S}_{k}^{+}$ denote 
the set of squarefree products of primes in $\mathfrak{L}_k^{+}$, including the so-called 
empty product corresponding to $1$. Now, for each vertex $\mathfrak{n} \in 
\mathfrak{S}_j^+$, we have a $p$-adic $L$-function 
$\mathcal{L}_{\mathfrak{p}}({\bf{f}}^{(\mathfrak{n})}, K_{\mathfrak{p}^{\infty}})$ or
$\mathcal{L}_{\mathfrak{p}}({\bf{f}}^{(\mathfrak{n})}, K_{\mathfrak{p}^{\infty}})^{\pm}$
in $\Lambda$ by the construction given above. As we explain below, each of these
$p$-adic $L$-functions is given by a product of completed group ring elements
$\theta_{{\bf{f}}^{\mathfrak{n}}} \theta_{{\bf{f}}^{\mathfrak{n}}}^{*}$, where 
$\theta_{{\bf{f}}^{\mathfrak{n}}} \in \Lambda$ is constructed in a natural way 
from the eigenform ${\bf{f}}^{\mathfrak{n}}$ via the Jacquet-Langlands 
correspondence and strong approximation at $\mathfrak{p}$, and 
$\theta_{{\bf{f}}^{\mathfrak{n}}}^{*} $ is the image of 
$\theta_{{\bf{f}}^{\mathfrak{n}}} $ under the involution of $\Lambda$
sending $\sigma$ to $\sigma^{-1}$ in $G_{\mathfrak{p}^{\infty}}$. Let
us write $\lambda_{\mathfrak{n}}$ to denote the completed group ring
element $\theta_{{\bf{f}}^{\mathfrak{n}}}$, which is only well defined up
to automorphism of $G_{\mathfrak{p}^{\infty}}$. We then obtain from Theorem 
\ref{IMC} the following result, following the discussion in \cite[Theorem 3.2.3 (c)]{Ho} 
(cf. also \cite[Theorem 1.3]{VO2}).

\begin{corollary} \label{IMCr1} Keep the setup of Theorem \ref{IMC}.
Suppose that for {\it{any}} height one prime $\mathfrak{Q}$ of $\Lambda$,
there exists an integer $k_0$ such that for all integers $j \geq k_0$ the set
\begin{align*} \lbrace \lambda_{\mathfrak{n}} \in \Lambda/(\mathfrak{P}^j): 
\mathfrak{n} \in \mathfrak{S}_{j}^{+} \rbrace\end{align*} contains at least one 
completed group ring element $\lambda_{\mathfrak{n}}$ with nontrivial image in 
$\Lambda/(\mathfrak{Q}, \mathfrak{P}^{k_0})$. Then, there is an equality of ideals
\begin{align*} \left(  \mathcal{L}_{\mathfrak{p}}({\bf{f}} , K_{\mathfrak{p}^{\infty}}) \right) =
 \left( \operatorname{char}_{\Lambda}\left( X({\bf{f}} , K_{\mathfrak{p}^{\infty}})  \right)  
\right) \text{ in } \Lambda. \end{align*} \end{corollary}

The results of Theorem \ref{IMC} and Corollary \ref{IMCr1} extend to the general setting of 
totally real fields, as explained in Theorem 1.3 of the sequel paper \cite{VO2}. We omit 
the statements here for simplicity of exposition. Finally, let us note that while we have not treated the construction for higher weights, the 
issue of Jochnowitz congruences (following Vatsal \cite{Va2} with Rajaei \cite{Raj}), or 
Howard's criterion itself in this note, these problems in fact motivate this work.

\begin{remark}[Notations.] Let ${\bf{A}}_F$ denote the adeles of $F$, with 
${\bf{A}} = {\bf{A}}_{\bf{Q}}.$ Let ${\bf{A}}_f$ denote the finite adeles
of ${\bf{Q}}.$ We shall sometimes write $T = \Res_{F/{\bf{Q}}}(K^{\times})$ 
to denote the algebraic group associated to $K^{\times}$, and 
$Z =  \Res_{F/{\bf{Q}}}(F^{\times})$ the algebraic group associated o $F^{\times}.$ 
Given a finite prime $v$ of $F$, fix a uniformizer $\varpi_v$ of $F_v$, and 
let $\kappa_v$ denote the residue field of $F_v$ at $v$, with $q = q_v$ its cardinality. 
 \end{remark}

\section{Some preliminaries}

\begin{remark}[Ring class towers.]

Given an ideal $\mathfrak{c} \subset \mathcal{O}_F$, let $\mathcal{O}_{\mathfrak{c}} = \mathcal{O}_F +
\mathfrak{c}\mathcal{O}_K$ denote the $\mathcal{O}_F$-order of conductor
$\mathfrak{c}$ in $K$. The  {\it{ring class field of conductor $\mathfrak{c}$ of $K$}} 
is the Galois extension $K[\mathfrak{c}]$ of $K$ characterized by class field 
theory via the identification \begin{align*}\begin{CD} \widehat{K}^{\times} / 
\widehat{\mathcal{O}}_{\mathfrak{c}}^{\times} K^{\times}
@>{\rec_K}>> \Gal(K[\mathfrak{c}]/K).\end{CD}\end{align*} Here, 
$\rec_K$ denotes reciprocity map, normalized to send uniformizers 
to their corresponding geometric Frobenius endomorphisms. Let $G[\mathfrak{c}]$ denote 
the Galois group $\Gal(K[\mathfrak{c}~]/K)$. Let $K[\mathfrak{p}^{\infty}] = \bigcup_{n\geq 0} 
K[\mathfrak{p}^n]$ denote the union of all ring class extensions of $\mathfrak{p}$-power conductor 
over $K$. Let us write $G[\mathfrak{p}^{\infty}]$ to denote the Galois group $\Gal(K[\mathfrak{p}^{\infty}]/K)$, 
which has the structure of a profinite group $G[\mathfrak{p}^{\infty}] = \varprojlim_n G[\mathfrak{p}^n].$

\begin{lemma}\label{Galois} The Galois group $G[\mathfrak{p}^{\infty}]$ has the
following properties. \begin{itemize}
\item[(i)] The reciprocity map $\rec_K$ induces an isomorphism of
topological groups \begin{align*}\begin{CD}
\widehat{K}^{\times}/K^{\times} U @>{\rec_K}>>
G[\mathfrak{p}^{\infty}],\end{CD}\end{align*}  where \begin{align*}
U = \bigcap_{n \geq 0}\widehat{\mathcal{O}}_{\mathfrak{p}^n}^{\times} = \lbrace x \in
\widehat{\mathcal{O}}_{K}^{\times}: x_{\mathfrak{p}} \in
\mathcal{O}_{F_{\mathfrak{p}}} \rbrace.\end{align*}

\item[(ii)] The torsion subgroup $G[\mathfrak{p}^{\infty}]_{\tors}$
of $G[\mathfrak{p}^{\infty}]$ is finite. Moreover, there is an isomorphism
of topological groups
\begin{align*}G[\mathfrak{p}^{\infty}]/G[\mathfrak{p}^{\infty}]_{\tors}
\longrightarrow {\bf{Z}}_{p}^{\delta},\end{align*} where $\delta =[F_{\mathfrak{p}}:
{\bf{Q}}_p].$\end{itemize}\end{lemma}

\begin{proof}
See \cite[$\S$ Lemma 2.1 and Corollary 2.2]{CV}. 
\end{proof} We shall use the following notations throughout.
Let $G_{\mathfrak{p}^{\infty}} = G[\mathfrak{p}^{\infty}]/
G[\mathfrak{p}^{\infty}]_{\tors}$ denote the
${\bf{Z}}_{p}^{\delta}$-quotient of $G[\mathfrak{p}^{\infty}]$. Let $K_{\mathfrak{p}^{\infty}}$ denote 
the dihedral or anticyclotomic ${\bf{Z}}_{p}^{\delta}$-extension of $K$, so that \begin{align*} 
G_{\mathfrak{p}^{\infty}} = \Gal(K_{\mathfrak{p}^{\infty}}/K) \cong
{\bf{Z}}_{p}^{\delta}.\end{align*} Given a positive integer $n$, we then let $K_{\mathfrak{p}^n}$ 
denote the extension of $K$ for which \begin{align*}G_{ \mathfrak{p}^n} = \Gal(K_{\mathfrak{p}^n}/K) \cong \left(
{\bf{Z}}/p^n{\bf{Z}}\right)^{\delta},\end{align*} so that 
$G_{\mathfrak{p}^{\infty}} = \varprojlim_n G_{\mathfrak{p}^n}$. \end{remark}

\begin{remark}[Central values.]

Here, we record the refinement of Waldspurger's theorem \cite{Wa}
given by Yuan-Zhang-Zhang \cite{YZ^2}, as well as the nonvanishing
theorem given by Cornut-Vatsal \cite{CV}. Let $B$ be a totally definite quaternion
algebra defined over $F$. We view the group of invertible elements $B^{\times}$ 
as an algebraic group with centre $Z$. We then view the group $K^{\times}$ as a 
maximal torus $T$ of $B^{\times}$ via a fixed embedding $K \longrightarrow B$. 
Fix an idele class character \begin{align*}\rho = 
\otimes_v \rho_v: {\bf{A}}_{K}^{\times}/K^{\times} \longrightarrow
{\bf{C}}^{\times}.\end{align*} Let $\pi = \otimes_v \pi_v$ be a cuspidal
automorphic representation of $\GL({\bf{A}}_F)$, assumed to have
trivial central character. Let $L(\pi, \rho, s)=\prod_v L(\pi_v,
\rho_v, s)$ denote the Rankin-Selberg $L$-function associated to
$\pi$ and $\rho$, with central value at $s=1/2$. Let $\epsilon(\pi_v, \rho_v, 1/2) 
\in \lbrace \pm1 \rbrace $ denote the local root number of $L(\pi, \rho, 1/2)$ at a 
prime $v$ of $F$. Let $\omega = \omega_{K/F}$ denote the quadratic character
associated to $K/F$. The set of places $v$ of $F$ given by
\begin{align*} \Sigma = \lbrace v: \epsilon(\pi_v, \rho_v, 1/2) \neq
\rho_v \cdot \omega_v (-1) \rbrace
\end{align*} is known to have finite cardinality, and the global
root number to be given by the product formula \begin{align*}
\epsilon(\pi, \rho, 1/2) = \prod_v \epsilon(\pi_v, \rho_v, 1/2) =
\left(-1 \right)^{|\Sigma|}.\end{align*} A theorem of Tunnell and
Saito gives a criterion to determine whether or not a given
prime of $F$ belongs to this set $\Sigma$. To state this
theorem, we must first say a word or two about the Jacquet-Langlands
correspondence. That is, the theorem of Jacquet and Langlands
\cite{JL} establishes an injection $ \Pi \longrightarrow \JL(\Pi)$ from 
the set of automorphic representations of $\left( B \otimes {\bf{A}}_F \right)^{\times}$ of
dimension greater than $1$ to the set of cuspidal automorphic
representations of $\GL({\bf{A}}_F)$. Moreover, it characterizes
the image of this injection as cuspidal automorphic
representations of $\GL({\bf{A}}_F)$ that are discrete
series (i.e. square integrable) at each prime $v \in \Ram(B)$. 
Let us write $\pi' = \prod_v \pi_{v}'$ to denote the Jacquet-Langlands 
correspondence of $\pi$ on $\left( B \otimes {\bf{A}}_F \right)^{\times}$, 
so that $\pi = \JL(\pi')$.

\begin{theorem}[Tunnell-Saito] \label{TS} If $\pi_v$ is a
discrete series for some prime $v$ of $F$, then let $\pi_{v}'$
denote the Jacquet-Langlands lift of $\pi_v$. Fix embeddings of 
algebraic subgroups $K_{v}^{\times} \rightarrow B_{v}^{\times}$ and $K_{v}^{\times}
\rightarrow \GL(F_v)$. We have that $v \in \Sigma$ if any only if 
$\Hom_{K_{v}^{\times}}(\pi_v \otimes \rho_v,{\bf{C}})=0.$ Moreover, we 
have that \begin{align}\label{ts} 
\operatorname{dim} \left( \Hom_{K_{v}^{\times}}(\pi_v \otimes
\rho_v, {\bf{C}}) \oplus \Hom_{K_{v}^{\times}}(\pi_{v}' \otimes \rho_v,
{\bf{C}}) \right) =1,\end{align} with the second space
$\Hom_{K_{v}^{\times}}(\pi_v' \otimes \rho_v, {\bf{C}})$ treated as
zero if $\pi_v$ is not discrete.
\end{theorem} \begin{proof} See \cite{Tu} and \cite{Sai}. 
\end{proof} Let us note that in the case where the root number 
$\epsilon(\pi, \rho, 1/2)$ equals $+1$, the set $\Sigma$ has even 
cardinality, in which case the ramification set of the quaternion algebra 
$B$ over $F$ is given exactly by $\Sigma$. Anyhow, we may now state 
the refinement of Waldspurger's theorem given by Yuan-Zhang-Zhang 
\cite[$\S$ 1.2]{YZ^2}. Given a vector $\Phi \in \pi'$, consider the period 
integral $l(\cdot, \rho): \pi' \rightarrow {\bf{C}}$ defined by
\begin{align}\label{periodintegral} l(\Phi, \rho) =
\int_ {T({\bf{A}}_F)/Z({\bf{A}}_F)T(F)} \rho (t) \cdot \Phi(t)
dt.\end{align} Here, $dt$ denotes the Tamagawa measure on
$T({\bf{A}}_F)$, which has volume $2 L(\omega, 1)$ on the
quotient $T({\bf{A}}_F)/ Z({\bf{A}}_F) T(F) $, and volume $2$ on the
quotient $\left( B \otimes{\bf{A}}_F \right)^{\times}/
{\bf{A}}_F^{\times} B^{\times} $. Let $\langle \cdot, \cdot \rangle$
denote the Petersson inner product on $\left( B \otimes{\bf{A}}_F
\right)^{\times}/{\bf{A}}_F^{\times} B^{\times}$ with respect to
$dt$. Fix a nontrivial hermitian form $\langle \cdot, \cdot
\rangle_v$ on $\pi_{v}'$ such that there is a product formula
$\langle \cdot, \cdot \rangle = \prod_v \langle \cdot, \cdot
\rangle_v$.

\begin{theorem}[Waldspurger, Yuan-Zhang-Zhang]\label{W}
Assume that $\Phi \in \pi'$ is nonzero and decomposable. Then for
any prime $v$ of $F$, the local functional defined by
\begin{align*}\alpha(\Phi_v, \rho_v) = \frac{ L(\omega_v, 1) \cdot L(\pi_v,
\operatorname{ad}, 1) }{\zeta_v(2) \cdot L(\pi_v, \rho_v, 1/2)}
\cdot \int_{K_{v}^{\times}/F_{v}^{\times}} \langle \pi_v'(t)\Phi_v,
\Phi_v\rangle_v \cdot \rho_v(t) d t \end{align*} does not vanish, and equals
$1$ for all but finitely many primes $v$ of $F$. Moreover, we have the identity
\begin{align}\label{value} \left| l(\Phi, \rho)\right|^2 =
\frac{\zeta_F(2) \cdot L(\pi, \rho, 1/2)}{2 \cdot L(\pi,
\operatorname{ad}, 1)} \cdot \prod_{v \nmid \infty} \alpha(\Phi_v,
\rho_v),\end{align} where the value $(\ref{value})$ is algebraic.
Here, $\zeta_F(s) = \prod_v \zeta_{F,v}(s)$ is the Dedekind zeta
function of $F$, and $L(\pi, \operatorname{ad}, s) = \prod_v L(\pi_v,
\operatorname{ad}, s)$ the $L$-series of the adjoint representation.
\end{theorem}

\begin{proof} See \cite{Wa} or \cite[Proposition 1.2.1]{YZ^2}. In
this setting, the value $(\ref{value})$ is known to lie in the field
${\bf{Q}}(\pi, \rho)$ generated by values of $\pi$ and $\rho$ as a
consequence of the fact that $B$ is totally definite (and hence compact). 
\end{proof} Let us now record the following nonvanishing theorem for the
central values $L(\pi, \rho, 1/2)$, whose proof relies on the
results of Tunnell-Saito and Waldspurger, as explained in the
introduction of \cite{CV}. Given an integer $n \geq 1$, let us call
a character that factors through $G[\mathfrak{p}^{n}]$
{\it{primitive of conductor $n$}} if it does not factor though
$G[\mathfrak{p}^{n-1}]$. Let $P(n, \rho_0)$ denote the set of
primitive characters on $G[\mathfrak{p}^n]$ with associated
(tamely ramified) character $\rho_0$ on $G[\mathfrak{p}^{\infty}]_{\tors}$.

\begin{theorem}[Cornut-Vatsal]\label{CV}
Assume that the root number $\epsilon(\pi, \rho, 1/2)$ is
$1$, and that the prime to $\mathfrak{p}$-part of the level of
$\pi$ is prime to the relative discriminant $\mathfrak{D}_{K/F}$.
Then, for all $n$ sufficiently large, there exists a primitive 
character $\rho \in P(n, \rho_{0})$ such that $L(\pi, \rho, 1/2) \neq 0$.
\end{theorem}  

\begin{proof}
See \cite[Theorem 1.4]{CV}, as well as the main result of
\cite{Va} for $F={\bf{Q}}$. \end{proof} \end{remark} We 
may then deduce from the algebraicity theorem of Shimura \cite{Sh2}, 
which applies to the values $L(\pi, \rho, 1/2)$, the following strengthening 
of this result.

\begin{corollary}\label{strongCV}
Assume that the root number $\epsilon(\pi, \rho, 1/2)$ is
$1$, and that the prime to $\mathfrak{p}$-part of the level of
$\pi$ is prime to the relative discriminant $\mathfrak{D}_{K/F}$.
Let $Y$ denote the set of all finite order characters of the Galois 
group $G[\mathfrak{p}^{\infty}]$, viewed as idele class characters
of $K$ via the reciprocity map $\rec_K$. Then, for all but finitely
many characters $\rho$ in $Y$, the central value $L(\pi, \rho, 1/2)$
does not vanish. \end{corollary}

\section{Modular forms on totally definite quaternion algebras}

We follow the exposition given in Mok \cite[$\S$2]{Mok}. Fix an
ideal $\mathfrak{N} \subset \mathcal{O}_F$ as defined in $(\ref{N})$ above. Fix a totally
imaginary quadratic extension $K/F$ of relative discriminant $\mathfrak{D}_{K/F}$ prime
to $\mathfrak{N}$, so that we have the factorization
$(\ref{factorization})$ of $\mathfrak{N}$ in $\mathcal{O}_F$. Let $B$ denote the
totally definite quaternion algebra over $F$ of discriminant
$\mathfrak{N}^{-}$.  Let us also fix isomorphisms $\iota_v: B_v \cong
M_2(F_v)$ for all primes $v \notin \Ram(B)$ of $F$. Note that $B$ splits 
at $\mathfrak{p}$ by our hypotheses on the level $\mathfrak{N}$.

\begin{remark}[Basic definition.] 

Let $\mathcal{O}$ be any ring. An {\it{$\mathcal{O}$-valued automorphic form of weight 
$2$, level $H$, and trivial central character on $B^{\times}$}} is a function \begin{align*}
\Phi: B^{\times}\backslash  \widehat{B}^{\times}/H \longrightarrow \mathcal{O}\end{align*} 
such that for all $g \in B^{\times}$, $b \in \widehat{B}^{\times}$, $h \in H$, and $z \in 
\widehat{F}^{\times}$, 
\begin{align}\label{transformation} \Phi(g b h z ) =\Phi(b). \end{align} 
Let $\mathcal{S}_2^B(H; \mathcal{O})$ denote the 
space of such functions, modulo those which factor through
the reduced norm homomorphism $\nrd$. A function in  
$\mathcal{S}_2^B(H; \mathcal{O})$ is called an 
{\it{$\mathcal{O}$-valued modular form
of weight $2$, level $H$, and trivial central character on $B^{\times}$.}}\end{remark}  

\begin{remark}[Choice of levels.] 

Fix $\mathfrak{M} \subset \mathcal{O}_F$ an integral ideal prime to the discriminant 
$\mathfrak{N}^{-}$ (we shall often just take $\mathfrak{M} = \mathfrak{p}\mathfrak{N}^{+}$). 
Given a finite prime $v \subset F$, let $R_v$ be a local order of $B_v$ such 
that \begin{align}\label{condition} \text{$R_v$ is {\it{maximal}} if $v
\mid \mathfrak{N}^{-}$, or {\it{Eichler of level
$v^{\ord_v(\mathfrak{M})}$}} if $v \nmid \mathfrak{N}^{-}$. }
\end{align} Write $\widehat{R} = \prod_v R_v$. Let $R = B \cap
\widehat{R}$ denote the corresponding Eichler order of $B$. Let $H_v
= R_{v}^{\times}$, so that $H = \prod_v H_v \subset
\widehat{B}^{\times}$ is the corresponding compact open subgroup. We
shall assume throughout that a compact open subgroup $H \subset
\widehat{B}^{\times}$ takes this form, in which case we refer to it as
{\it{$\mathfrak{M}$-level structure on $B$}}. \end{remark}

\begin{remark}[Hecke operators.]

We define Hecke operators acting on the space $S_2^B(H; \mathcal{O})$.

\begin{remark}[The operators $T_v.$] 

Given any finite prime $v$ of $F$ that splits $B$ and does not divide the level of $H$, 
we define Hecke operators $T_v$ as follows. Note that since $R_v \subset B_v$ is 
maximal by $(\ref{condition})$, we have the identification $\iota_{v}: R_v^{\times} \cong
\GL(\mathcal{O}_{F_v})$. Suppose now that we have any double coset
decomposition \begin{align}\label{T_vdecomp}\GL(\mathcal{O}_{F_v}) 
\left(\begin{array} {cc} \varpi_v& 0 \\
0 & 1 \end{array}\right)  \GL(\mathcal{O}_{F_v}) = \coprod_{a \in
{\bf{P}}^1(\kappa_v)} \sigma_a~\GL(\mathcal{O}_{F_v}).\end{align} The Hecke
operator $T_v$ is then defined by the rule \begin{align*}\left( T_v \Phi \right)(b) =
\sum_{a \in {\bf{P}}^1(\kappa_v)} \Phi\left( b \cdot
\iota_{v}^{-1}\left( \sigma_a \right) \right).\end{align*}
An easy check with the transformation property $(\ref{transformation})$ 
shows that these definitions do not depend on choice of representatives 
$\lbrace \sigma_a(v)\rbrace_{a \in {\bf{P}}^1(\kappa_v)}$.\end{remark}

The set of representatives 
$\lbrace \sigma_a \rbrace_{a \in \kappa_{\mathfrak{p}}}$ in $(\ref{T_vdecomp})$
has the following lattice description. Let $q = q_v$ denote the cardinality
of $\kappa_v$. Let $\lbrace L(a) \rbrace_{a \in \kappa_v}$ 
denote the set of $q+1$ sublattices of
$\mathcal{O}_{F_v} \oplus \mathcal{O}_{F_v}$ of index equal to $q$.            
Arrange the matrix representatives $\sigma_a$ so that $\sigma_a
\left( \mathcal{O}_{F_v} \oplus
\mathcal{O}_{F_v}\right) = L(a)$. The set of lattices
$\lbrace L(a) \rbrace_{a \in {\bf{P}}^1(\kappa_v)}$ then describes
the set of representatives. 

\begin{remark}[The operator $U_{\mathfrak{p}}.$]  

In general, given any finite prime $v$ of $F$ and any integer $m \geq 1$,
we let $I_{v^m}$ denote the Iwahori subgroup of level $v^m$ of 
$\GL(\mathcal{O}_{F_v})$, \begin{align*}I_{v^m} =
\lbrace \left(\begin{array} {cc} a & b \\
c & d \end{array}\right) \in \GL(\mathcal{O}_{F_v}): c \equiv 0
\mod \varpi_{v}^m \rbrace.\end{align*} Let us now suppose that 
$v$ is a finite prime of $F$ not dividing $\mathfrak{N}^{-}$ where
the level is not maximal, though we shall only be interested in the 
special case of $v = \mathfrak{p}$. Hence, taking $v = \mathfrak{p}$, let us assume 
that the level $H = \prod_v R_v^{\times}$ is chosen so that $\iota_{\mathfrak{p}} :
R_{\mathfrak{p}}^{\times} \cong I_{\mathfrak{p}}.$ Suppose we
have any double coset decomposition \begin{align}\label{U_pdecomp}
I_{\mathfrak{p}} \left(\begin{array} {cc} 1 & 0 \\
0 & \varpi_{ \mathfrak{p} } \end{array}\right) I_{\mathfrak{p}} =
\coprod_{a \in \kappa_{\mathfrak{p}}} \sigma_a ~ I_{\mathfrak{p}}.\end{align}
The operator $U_{\mathfrak{p}}$ is then defined by the rule \begin{align*}\left(
U_{\mathfrak{p}} \Phi \right)(b) = \sum_{a \in \kappa_{\mathfrak{p}}
} \Phi\left( b \cdot
\iota_{\mathfrak{p}}^{-1}(\sigma_a)\right).\end{align*} Another easy check
with the transformation property $(\ref{transformation})$ shows that
this definition does not depend on choice of representatives
$\lbrace \sigma_a \rbrace_{a \in \kappa_{\mathfrak{p}}}$. 

The set of representatives $\lbrace \sigma_a \rbrace_{a \in \kappa_{\mathfrak{p}}}$ 
in $(\ref{U_pdecomp})$ has the following lattice description. Recall 
that we fix a uniformizer $\varpi_{\mathfrak{p}}$ of $\mathfrak{p}$, and let $q =
q_{\mathfrak{p}}$ denote the cardinality of the residue field
$\kappa_{\mathfrak{p}}$. Let $\lbrace L(a) \rbrace_{a \in
\kappa_{\mathfrak{p}}}$ denote the set of sublattices of index
$q$ of $\mathcal{O}_{F_{\mathfrak{p}}} \oplus \varpi_{\mathfrak{p}}
\mathcal{O}_{F_{\mathfrak{p}}}$, minus the sublattice
$\varpi_{\mathfrak{p}} \left( \mathcal{O}_{F_{\mathfrak{p}}} \oplus
\mathcal{O}_{F_{\mathfrak{p}}} \right).$ (Note that there are exactly
$q$ lattices in this set, as there are exactly $q+1$
sublattices of ``distance $1$'' away from
$\mathcal{O}_{F_{\mathfrak{p}}} \oplus \varpi_{\mathfrak{p}}
\mathcal{O}_{F_{\mathfrak{p}}}$  -- see \cite[$\S$ II.2]{Vi} or
$(\ref{distance})$ below). We consider translates of 
the pair \begin{align*} \left( \mathcal{O}_{F_{\mathfrak{p}}} \oplus
\mathcal{O}_{F_{\mathfrak{p}}},  \mathcal{O}_{F_{\mathfrak{p}}} \oplus
\varpi_{\mathfrak{p}}\mathcal{O}_{F_{\mathfrak{p}}} \right),\end{align*}
which is stabilized by the Iwahori subgroup $I_{\mathfrak{p}}$.
Let us now arrange the representatives $\lbrace \sigma_a 
\rbrace_{a \in \kappa_{\mathfrak{p}}}$ so that \begin{align*}
\sigma_a \left(\mathcal{O}_{F_{\mathfrak{p}}} \oplus
\mathcal{O}_{F_{\mathfrak{p}}} \right) &=  
\mathcal{O}_{F_{\mathfrak{p}}} \oplus \varpi_{\mathfrak{p}}
\mathcal{O}_{F_{\mathfrak{p}}} = L(a)\\ 
\sigma_a \left(\mathcal{O}_{F_{\mathfrak{p}}} \oplus
\varpi_{\mathfrak{p}}\mathcal{O}_{F_{\mathfrak{p}}} \right) 
&= L'(a).\end{align*} We shall then consider 
the pairs of lattice translates given by $\left( L(a), 
L'(a)\right)$. \end{remark} \end{remark}

\begin{remark}[Strong approximation.]

Let $F_{+}$ denote the totally positive elements of $F$, i.e. the
elements of $F$ whose image under any real embedding of $F$ is
greater than $0$. Let $\lbrace x_i \rbrace_{i=1}^{h}$ be
a set of representatives for the narrow class group
\begin{align*} \mathfrak{Cl}_F = F_{+}^{\times} \backslash \widehat{F}^{\times}
/\widehat{\mathcal{O}}_{F}^{\times} \end{align*} of $F$, such that $\left(
x_i\right)_{\mathfrak{p}} = 1$ for each $i = 1, \ldots h$.

\begin{lemma} Given $H \subset \widehat{B}^{\times}$ any compact open subgroup,
there is a bijection \begin{align} \label{SA}
\coprod_{i=1}^{h} B^{\times} \xi_i
B_{\mathfrak{p}}^{\times} H  \cong \widehat{B}^{\times}.
\end{align} Here, each $\xi_i$ is an element of 
$\widehat{B}^{\times}$ such that $(\xi_i)_{\mathfrak{p}}=1$ and $\nrd(\xi_i) = x_i$.\end{lemma}

\begin{proof}
This is a standard result, it can be deduced from the strong
approximation theorem (\cite[$\S$ III.4, Th\'eor\`eme 4.3]{Vi})
applied to the norm theorem (\cite[$\S$ III.4, Th\'eor\`eme 4.1]{Vi}). \end{proof} 
Let us now fix a set of representatives $(x_i)$ for the modified class group 
$\mathfrak{Cl}_F/F_{\mathfrak{p}}^{\times}$ such that $(x_i)_{\mathfrak{p}}=1$ 
for each $i=1, \ldots, h$. Let us then choose $\xi_i \in \widehat{B}^{\times}$ such
that $(\xi_i)_{\mathfrak{p}}=1$ and $\nrd(\xi_i)=x_i$ for each $i$.
Now, for each $i$, let us define a subgroup \begin{align}\label{gamma} 
\Gamma_{i} = \lbrace b \in B^{\times}:
b_v \in \left( \xi_{i,v}\right) H_v \left( \xi _{i, v}\right)^{-1} 
\text{ for all $v \nmid \mathfrak{p}$} \rbrace \subset B^{\times}.\end{align} 
Observe that each $\Gamma_{i}$ embeds discretely into
$B_{\mathfrak{p}}^{\times}$. (That is, each $\Gamma_i$ can be described 
as the intersection of $B^{\times}$ with some conjugate of $H$. This
intersection can then be embedded discretely into $B_{S}^{\times}
= \prod_{v \in S} B_v^{\times}$, where $S$ denotes the set of archimedean 
places $\Sigma_F$ of $F$ along with the prime $\mathfrak{p}$. On the 
other hand, recall that $B$ is ramified at all of the archimedean places 
of $F$, and hence that $B^{\times}_{\Sigma_F}$ is compact. It is then follows from a 
standard fact in the theory of topological groups that the intersection
defining each $\Gamma_i$ embeds discretely into $B_{S}^{\times}$ modulo 
the compact subgroup $B_{\Sigma_F}^{\times}$. Hence, each $\Gamma_i$ 
embeds discretely into $ B_{\mathfrak{p}}^{\times}$). We may therefore 
view each $\Gamma_i$ as a discrete subgroup of $\GL(F_{\mathfrak{p}})$ 
via our fixed isomorphism $\iota_{\mathfrak{p}}: B_{\mathfrak{p}}^{\times} 
\cong \GL(F_{\mathfrak{p}})$.

\begin{corollary} We have a bijection
\begin{align}\label{SA2} \coprod_{i=1}^{h} \Gamma_{i}\backslash
B_{\mathfrak{p}}^{\times}/H_{\mathfrak{p}} &\cong
B^{\times} \backslash \widehat{B}^{\times}/H \end{align} via the map
given on each component by $[g] \mapsto [\xi_i \cdot g].$ That is, for
$g \in B_{\mathfrak{p}}^{\times}$, the class of $g$ in each
component $\Gamma_i \backslash B_{\mathfrak{p}}^{\times}/
H_{\mathfrak{p}}$ goes to the class of $\xi_i \cdot g $ in $B^{\times}
\backslash \widehat{B}^{\times}/H.$ \end{corollary} 

Fix $\Phi \in S_2^B(H; \mathcal{O})$. Via a fixed bijection $(\ref{SA2})$, 
we may view $\Phi$ as an $h$-tuple of functions $\left( \phi^i 
\right)_{i=1}^{h}$ on $\GL(F_{\mathfrak{p}})$ that 
satisfy the relation \begin{align}\label{relation} 
\phi^i(\gamma b h z ) = \phi^i (b) \end{align} for all $\gamma \in
\Gamma_{i}$, $b \in B_{\mathfrak{p}}^{\times} \cong \GL(F_{\mathfrak{p}})$, 
$h \in H_{\mathfrak{p}}$, and $z \in \widehat{F}^{\times}_{\mathfrak{p}}$. 
The identification $(\ref{SA2})$ then allows us to describe 
$\Phi \in S_2^B(H;\mathcal{O})$ as a vector of functions on homothety
classes of full rank lattices of $F_{\mathfrak{p}} \oplus F_{\mathfrak{p}}$ in the following way.
Let $\mathcal{L}\left(F_{\mathfrak{p}}\oplus F_{\mathfrak{p}}\right)$ denote 
the set of homothety classes of full rank lattices of $F_{\mathfrak{p}} \oplus F_{\mathfrak{p}}$.

\begin{remark}[Case I: $\mathfrak{p}\nmid\mathfrak{M}.$]
Fix $\Phi \in \mathcal{S}_2^B(H ; \mathcal{O})$ as above, associated to a 
vector of functions $\left(\phi^i \right)_{i=1}^{h}$. For
each $\phi^i$, we can define a corresponding function 
$c_{\phi^i}$ on $\mathcal{L}\left( F_{\mathfrak{p}} 
\oplus F_{\mathfrak{p}}\right)$ as follows. Given a class
$[L] \in \mathcal{L} \left( F_{\mathfrak{p}} \oplus F_{\mathfrak{p}}\right)$ with
fixed representative $L$, let $g_L \in \GL(F_{\mathfrak{p}})$ be any matrix such that 
$g_L\left( \mathcal{O}_{F_{\mathfrak{p}} } \oplus
\mathcal{O}_{F_{\mathfrak{p}} }\right) = L$. Let
\begin{align*}c_{\phi^i}([L]) = \phi^i(g_L).\end{align*} 
This definition does not depend on choice of matrix representative $g_L$, as 
a simple check using relation $(\ref{relation})$ reveals. 
It also follows from $(\ref{relation})$ that 
\begin{align}\label{relation2} c_{\phi^i}([\gamma L]) =
c_{\phi^i}([L])\end{align} for all $\gamma \in \Gamma_i,$
for instance by taking $g_{\gamma L} = \gamma  g_L$, then seeing
that ${\phi}^i(g_{\gamma L}) = {\phi}^i(\gamma g_L ) = 
\phi^i(g_L)$.\end{remark}

\begin{remark}[Case II: $\mathfrak{p} \mid \mathfrak{M}$.] 
Consider $\Phi \in \mathcal{S}_2^B(H; \mathcal{O})$ as above, associated 
to a vector of functions $ \left(\phi^i \right)_{i=1}^{h}$. 
For each $\phi^i$, we can define a corresponding function 
$c_{\phi^i}$ on pairs $([L_1], [L_2])$ of classes in $\mathcal{L}\left( 
F_{\mathfrak{p}} \oplus F_{\mathfrak{p}}\right)$ (with a fixed pair of representatives 
$(L_1, L_2)$) satisfying the property that $L_2 \subset L_1$ with index $q$. That is, 
fix such a pair of classes $([L_1], [L_2])$. Let $g_L \in \GL(F_{\mathfrak{p}})$ 
be any matrix such that \begin{align*} g_L\left( \mathcal{O}_{F_{\mathfrak{p}}}
\oplus \mathcal{O}_{F_{\mathfrak{p}}}\right) &= L_1 \\
g_L \left(\mathcal{O}_{F_{\mathfrak{p}}} \oplus \varpi_{\mathfrak{p}} 
\mathcal{O}_{F_{\mathfrak{p}}} \right) &= L_2.\end{align*}
Let \begin{align*} c_{\phi^i}([L_1], [L_2]) = 
\phi^i(g_L).\end{align*} As before,
a simple check using relation $(\ref{relation})$ reveals that
$c_{\phi^i}([L_1], [L_2])$ does not depend on choice of matrix representative 
$g_L$. It also follows from $(\ref{relation})$ that \begin{align} 
c_{\phi^i}([\gamma L_1], [\gamma L_2]) = c_{\phi^i}([L_1], [L_2])\end{align} 
for all $\gamma \in \Gamma_i.$\end{remark} Let us write ${\bf{c}}_{\Phi}$ 
to denote the vector of lattice class functions $\left( c_{\phi^i} \right)_{i=1}^{h}$ associated to 
$\left({\phi}^i \right)_{i=1}^{h}$ in either case on the level 
$\mathfrak{M}$. We then have the following description of Hecke operators for 
these functions:

\begin{remark}[Case I: $\mathfrak{p} \nmid \mathfrak{M}.$]
The Hecke operator $T_{\mathfrak{p}}$ on ${\bf{c}}_{\Phi}$ is given by
\begin{align}\label{latticeT_p} c_{T_{\mathfrak{p}}\Phi} ([L])
= \sum_{L' \subset L} c_{\Phi}([L']).\end{align} Here, fixing representatives, the 
sum runs over the $q+1$ sublattices $L'$ of $L$ having
index equal to $q$.\end{remark}

\begin{remark}[Case II: $\mathfrak{p} \mid \mathfrak{M}$.]
The Hecke operator $U_{\mathfrak{p}}$ on ${\bf{c}}_{\Phi}$ is 
given by  \begin{align} \label{latticeU_p} 
 c_{U_{\mathfrak{p}}\Phi} ([L_1], [L_2]) = \sum_{L' \subset L_2} 
c_{\Phi}([L_2], [L']).\end{align} Here, fixing representatives, the sum runs over the
sublattices $L' \subset L_2$ of index $q$,
minus the sublattice $L'$ corresponding to 
$\varpi_{\mathfrak{p}}L_2$.\end{remark}
\end{remark}

\begin{remark}[The Bruhat-Tits tree of $\PGL(F_{\mathfrak{p}})$.]

The description above of modular forms $\Phi \in S_{2}^{B}(H; \mathcal{O})$ 
as functions ${\bf{c}}_{\Phi}$ on the set of homothety classes 
$\mathcal{L}\left( F_{\mathfrak{p}} \oplus F_{\mathfrak{p}} \right)$
has the following combinatorial interpretation. To fix ideas, 
let $\mathcal{M}(\M(F_{\mathfrak{p}}))$ denote the set of maximal
orders of $\M(F_{\mathfrak{p}})$. The group $\PGL(F_{\mathfrak{p}}) \cong
B_{\mathfrak{p}}^{\times}/F_{\mathfrak{p}}^{\times}$ acts on by
conjugation on $\mathcal{M}(\M(F_{\mathfrak{p}}))$. 

\begin{lemma}\label{transitive}
The conjugation action of $\PGL(F_{\mathfrak{p}})$ on
$\mathcal{M}(\M(F_{\mathfrak{p}}))$ is transitive. Moreover, there is an
identification of $\PGL(F_{\mathfrak{p}})/\PGL(\mathcal{O}_{F_{\mathfrak{p}}})$
with $\mathcal{M}(\M(F_{\mathfrak{p}}))$.
\end{lemma}

\begin{proof} This is a standard result, see for instance
\cite[$\S$ II.2]{Vi}. \end{proof} Let $\mathcal{T}_{\mathfrak{p}} =
(\mathcal{V}_{\mathfrak{p}}, \mathcal{E}_{\mathfrak{p}})$ denote the
{\it{Bruhat-Tits tree of $B_{\mathfrak{p}}^{\times}/F_{\mathfrak{p}}^{\times} \cong
\PGL(F_{\mathfrak{p}})$}}, by which we mean the tree of maximal
orders of $B_{\mathfrak{p}} \cong \M(F_{\mathfrak{p}}) $, such that

\begin{itemize}

\item[(i)] The vertex set $\mathcal{V}_{\mathfrak{p}}$ is indexed by maximal
orders of $\M(F_{\mathfrak{p}})$.

\item[(ii)] The edgeset $\mathcal{E}_{\mathfrak{p}}$ is indexed by Eichler
orders of level $\mathfrak{p}$ of $\M(F_{\mathfrak{p}}).$

\item[(iii)] The edgeset $\mathcal{E}_{\mathfrak{p}}$ has an {\it{orientation}}, 
i.e. a pair of maps \begin{align*}s, t: \mathcal{E}_{\mathfrak{p}} \longrightarrow
\mathcal{V}_{\mathfrak{p}}, ~~~ \mathfrak{e} \mapsto
(s(\mathfrak{e}), t(\mathfrak{e}))\end{align*} that assigns to each edge
$\mathfrak{e} \in \mathcal{E}_{\mathfrak{p}}$ a {\it{source}}
$s(\mathfrak{e})$ and a {\it{target}} $t(\mathfrak{e})$. Once such a
choice of orientation is fixed, let us write
$\mathcal{E}_{\mathfrak{p}}^{*}$ to denote the ``directed" edgeset
of $\mathcal{T}_{\mathfrak{p}}$.\end{itemize} Hence, we obtain
from Lemma \ref{transitive} the following immediate

\begin{corollary} The induced conjugation action of $\PGL(F_{\mathfrak{p}})$ on
$\mathcal{T}_{\mathfrak{p}}$ is transitive. Moreover, there is an 
identification of
$\PGL(F_{\mathfrak{p}})/\PGL(\mathcal{O}_{F_{\mathfrak{p}}})$ with 
$\mathcal{V}_{\mathfrak{p}}$.\end{corollary} Now, recall that we let 
$\mathcal{L}(F_{\mathfrak{p}}\oplus F_{\mathfrak{p}})$ denote the
set of homothety classes of full rank lattices of $F_{\mathfrak{p}} \oplus
F_{\mathfrak{p}}$.

\begin{lemma}\label{lattices-orders} There is a bijection $\mathcal{L}(F_{\mathfrak{p}}
\oplus F_{\mathfrak{p}}) \cong \mathcal{M}(\M(F_{\mathfrak{p}}).$ 
\end{lemma}

\begin{proof} See \cite[II $\S 2$]{Vi}. Let $V = F_{\mathfrak{p}}\oplus
F_{\mathfrak{p}}$, viewed as a $2$-dimensional $F_{\mathfrak{p}}$-vector
space. Fixing a basis $\lbrace z_1, z_2 \rbrace$ of $V$, we obtain the 
standard identification \begin{align*} \M(F_{\mathfrak{p}}) &\longrightarrow
\End_{F_{\mathfrak{p}}}(V) \\ \gamma &\longmapsto \left( v \mapsto \gamma
\cdot v \right).\end{align*} Here, $v = xz_1 + yz_2$ denotes any element of 
$V$, and $\gamma \cdot v$ the usual matrix operation of $\gamma$ 
on $(x, y)$. Hence, we obtain a bijective correspondence between 
maximal orders of $\M(F_{\mathfrak{p}})$ and maximal orders of 
$\End_{F_{\mathfrak{p}}}(V)$. Now, it is well known that the maximal
orders of $\End_{F_{\mathfrak{p}}}(V)$ take the form of 
$\End_{F_{\mathfrak{p}}}(L)$, with $L$ any full rank lattice of $F_{\mathfrak{p}}
\oplus F_{\mathfrak{p}}$, as shown for instance in \cite[II $\S$ 2, Lemme 2.1(1)]{Vi}.
Since the rings $\End_{F_{\mathfrak{p}}}(L)$ correspond in a natural way to 
classes $[L] \in \mathcal{L}(F_{\mathfrak{p}} \oplus F_{\mathfrak{p}})$, the claim follows. 
\end{proof} Hence, fixing an isomorphism $\mathcal{L}(F_{\mathfrak{p}}
\oplus F_{\mathfrak{p}}) \cong \mathcal{M}(\M(F_{\mathfrak{p}})$, we can associate
to each class $[L] \in \mathcal{L}(F_{\mathfrak{p}} \oplus F_{\mathfrak{p}})$ 
a unique vertex $\mathfrak{v}_{[L]} \in \mathcal{V}_{\mathfrak{p}}$. Let us also for future
reference make the following definition (cf. \cite[$\S$ II.2]{Vi}). Let 
$[L], [L'] \in \mathcal{L}(F_{\mathfrak{p}}\oplus F_{\mathfrak{p}})$ 
be a pair of classes with fixed representatives $(L, L')$ such that $L \supset L'$. Fix 
bases $(z_1, z_2)$ and $(\pi_{\mathfrak{p}}^a z_1, \pi_{\mathfrak{p}}^b z_2 )$
for $L$ and $L'$ respectively. We define the
{\it{distance}} $d(\mathfrak{v}_{[L]}, \mathfrak{v}_{[L']})$ between
the associated vertices $\mathfrak{v}_{[L]}$ and $ \mathfrak{v}_{[L']}$ 
by \begin{align}\label{distance}
d(\mathfrak{v}_{[L]}, \mathfrak{v}_{[L']}) = \left| b - a \right|.\end{align}
This definition does not depend on choice of representatives or bases. Let
$\mathfrak{v}_0$ denote the vertex of $\mathcal{V}_{\mathfrak{p}}$
corresponding to the maximal order
$\M(\mathcal{O}_{F_{\mathfrak{p}}})$. The {\it{length}} of any
vertex $\mathfrak{v} \in \mathcal{V}_{\mathfrak{p}}$ is then given by 
the distance $d(\mathfrak{v}, \mathfrak{v}_0)$. 

Now, the subgroups $\Gamma_i \subset B_{\mathfrak{p}}^{\times}$ defined in $(\ref{gamma})$ 
above act naturally by conjugation on $\mathcal{T}_{\mathfrak{p}}$, and the quotient graphs 
$\Gamma_i \backslash \mathcal{T}_{\mathfrak{p}}$ are finite. We may therefore consider the 
disjoint union of finite quotient graphs \begin{align*}\coprod_{i = 1}^{h} \Gamma_i \backslash 
\mathcal{T}_{\mathfrak{p}} = \left(\coprod_{i = 1}^{h} \Gamma_i \backslash
\mathcal{V}_{\mathfrak{p}}, \coprod_{i = 1}^{h} \Gamma_i
\backslash \mathcal{E}_{\mathfrak{p}}^{*} \right).\end{align*} Moreover, we may consider
the following spaces of modular forms defined on these disjoint unions of finite quotient graphs.
\begin{definition} Given a ring $\mathcal{O}$, let $\mathcal{S}_2\left( \coprod_{i=1}^{h} \Gamma_i
\backslash \mathcal{T}_{\mathfrak{p}} ; \mathcal{O} \right)$ denote the space
of vectors $\left( {\phi}^i  \right)_{i=1}^h$ of $\mathcal{O}$-valued, $\left(\Gamma_i
\right)_{i=1}^{h}$-invariant functions on $\mathcal{T}_{\mathfrak{p}}
= \left( \mathcal{V}_{\mathfrak{p}}, \mathcal{E}_{\mathfrak{p}}^{*}
\right).$ Here, it is understood that $\Phi \in \mathcal{S}_2\left(
\coprod_{i=1}^{h}\Gamma_i \backslash
\mathcal{T}_{\mathfrak{p}}; \mathcal{O} \right)$ is a function on
$\coprod_{i=1}^{h}\Gamma_i \backslash
\mathcal{V}_{\mathfrak{p}}$ if $\mathfrak{p} \nmid \mathfrak{M}$,
or a function on $\coprod_{i=1}^{h}\Gamma_i \backslash
\mathcal{E}_{\mathfrak{p}}^{*}$ if $\mathfrak{p} \mid \mathfrak{M}$.
\end{definition} 

\begin{proposition}\label{SA3} 
Let $\mathcal{O}$ be any ring. We have a bijection of
spaces \begin{align}\label{bij}  \mathcal{S}_2\left(
\coprod_{i=1}^{h} \Gamma_i \backslash
\mathcal{T}_{\mathfrak{p}} ; \mathcal{O} \right) &\cong
\mathcal{S}_{2}^{B}(H; \mathcal{O}).\end{align}
\end{proposition}

\begin{proof} Fix a function $\Phi \in \mathcal{S}_2^B(H;\mathcal{O})$.
Recall that by $(\ref{SA2})$, we have a bijection 
\begin{align*} \coprod_{i = 1}^{h}\Gamma_i \backslash 
B_{\mathfrak{p}}^{\times}/H_{\mathfrak{p}} &\longrightarrow B^{\times} 
\backslash \widehat{B}^{\times}/H,\end{align*} and moreover that we can 
associate to $\Phi$ a vector of functions ${\bf{c}}_{\Phi}= \left( c_{{\phi}^i}\right)$ 
on the set of homothety classes $\mathcal{L}(F_{\mathfrak{p}}\oplus F_{\mathfrak{p}}).$ The
vector ${\bf{c}}_{\Phi}$ is then clearly determined uniquely by the transformation
law for $\Phi$ in view of this bijection. Hence, by Lemma
\ref{lattices-orders}, we may view ${\bf{c}}_{\Phi}$ as a vector of 
functions on the set of maximal orders $\mathcal{M}(\M(F_{\mathfrak{p}}))$. 
Since we saw above that each of the functions $c_{{\phi}^i}$ is 
$\Gamma_i$-invariant, the claim follows. \end{proof} 
Let us for ease of notation write $\Phi$ to denote both a function in the space of 
modular forms $\mathcal{S}_2^B(H;\mathcal{O}) $, as well as its corresponding 
vector of functions $c_{\Phi}$ on maximal orders of $\M(F_{\mathfrak{p}})$ 
in $\mathcal{S}_2\left(\coprod_{i=1}^{h} \Gamma_i \backslash 
\mathcal{T}_{\mathfrak{p}} ; \mathcal{O} \right).$ We can then write down
the following combinatorial description of the Hecke operators $T_{\mathfrak{p}}$
and $U_{\mathfrak{p}}$ defined above, dividing into cases on the level structure.

\begin{remark} [Case I: $\mathfrak{p} \nmid \mathfrak{M}$.] We obtain
the following description of the operator $T_{\mathfrak{p}}$. Let
$\Phi({\bf{\mathfrak{v}}})$ denote the $h$-tuple of
functions $\left(c_{\phi^i}(\mathfrak{v})
\right)_{i=1}^{h} $ evaluated at a fixed vertex $\mathfrak{v}$
of $\mathcal{V}_{\mathfrak{p}}$. By $(\ref{latticeT_p})$, we obtain the description
\begin{align} \left( T_{\mathfrak{p}} \Phi \right) ({\bf{\mathfrak{v}}}) =
\sum_{ \mathfrak{w} \rightarrow \mathfrak{v} }
c_{\Phi}(\mathfrak{w}).
\end{align} Here, the sum ranges over all $q+1$ vertices
$\mathfrak{w}$ adjacent to $\mathfrak{v}$. \end{remark}

\begin{remark} [Case II: $\mathfrak{p} \mid \mathfrak{M}$.] We obtain
the following description of the operator $U_{\mathfrak{p}}$. Let
$\Phi({\bf{\mathfrak{e}}})$ denote the $h$-tuple of
functions $\left(c_{\phi^i}(\mathfrak{e})\right)_{i=1}^{h} $ evaluated 
at a fixed edge $\mathfrak{e}$. By $(\ref{latticeU_p})$, we obtain the description
\begin{align} \left( U_{\mathfrak{p}} \Phi \right) ({\bf{\mathfrak{e}}}) =
\sum_{s({\bf{\mathfrak{e}}}') = t({\bf{\mathfrak{e}}})}
c_{\Phi}({\bf{\mathfrak{e}}}'). \end{align} Here, the sum runs over
the $q$ edges $\mathfrak{e}' \in \mathcal{E}_{\mathfrak{p}}^{*}$
such that $s(\mathfrak{e}') = t(\mathfrak{e})$, minus the edge 
$\overline{\mathfrak{e}}$ obtained by reversing orientation of $\mathfrak{e}$. \end{remark}
\end{remark}

\begin{remark}[The Jacquet-Langlands correspondence.]

We now give a combinatorial version of the theorem of Jacquet and 
Langlands \cite{JL} under the bijection $(\ref{bij})$. Let us first review 
some background from the theory of Hilbert modular forms. We refer 
the reader to \cite{Ga}, \cite{Ge} or \cite{Go} for basic definitions
and background.

\begin{remark}[Hilbert modular forms.] Let ${\bf{f}} \in \mathcal{S}_2
(\mathfrak{N})$ be a cuspidal Hilbert modular form of parallel weight two, 
level $\mathfrak{N} \subset \mathcal{O}_F$, and trivial Nebentypus with 
associated vector of functions $\left( f_i \right)_{i = 1}^{h}$ on the $d$-fold 
product $\mathfrak{H}^d$ of the complex upper-half plane. We write 
$\mathcal{S}_2(\mathfrak{N})$ to denote this space of functions, which 
comes equipped with an action by classically or adelically defined Hecke 
operators for each finite prime $v$ of $F$, which we denote by $T_v$ in 
an abuse of notation. Let us also write $U_v$ to denote these operators 
when $v$ divides $\mathfrak{N}$. Fix a set of representatives 
$\left( t_i \right)_{i = 1}^{h}$ for the narrow
class group $\mathfrak{Cl}_F$. By the weak approximation 
theorem, we may choose these representatives in such way that 
$\left( t_i \right)_{\infty}=1$ for each $i$. Given a vector 
$z = (z_1, \ldots, z_d) \in F \otimes {\bf{R}}$, let us also 
define operations \begin{align*}\operatorname{Tr}(z)
= \sum_{i=1}^d z_i, ~~~ \mathcal{N}(z) = \prod_{i=1}^dz_i. 
\end{align*} Let us then define $e_F(z) = \exp(2\pi i 
\operatorname{Tr}(z)).$ It can be deduced from the transformation
law satisfied by ${\bf{f}}$ that each $f_i$ admits a Fourier series
expansion \begin{align*}f_i(z) = \sum_{\mu \in t_i \atop \mu \gg o} 
a_i(\mu)e_F(\mu z),\end{align*} where the sum over $t_i$ means 
the sum over ideals generated by the id\`eles $t_i$, and
$\mu \gg 0 $ means that $\mu$ is strictly positive. 
Now, any fractional ideal $\mathfrak{m}$ of $F$ can be written
uniquely as some $(\mu)t_i^{-1}$ with $\mu \in t_i$ totally positive.
We use this to define a normalized Fourier coefficient 
$a_{\mathfrak{m}}({\bf{f}})$ of ${\bf{f}}$ at $\mathfrak{m}$ in the 
following way. That is, let
\begin{align*} a_{\mathfrak{m}}({\bf{f}}) &= \begin{cases}
a_i(\mu) \cdot \mathcal{N}(t_i)^{-1} &\text{ if $\mathfrak{m}$ is 
integral}\\ 0 &\text{ else}.\end{cases}\end{align*} Equivalently,
we can define the normalized Fourier coefficient 
$a_{\mathfrak{m}}({\bf{f}})$ using the adelic Fourier expansion of ${\bf{f}}$.
That is, let $\vert \cdot \vert_{{\bf{A}}_F}$ denote the adelic norm, and
let $\chi_F: {\bf{A}}_F/F \longrightarrow {\bf{C}}^{\times}$ denote the 
standard additive character whose restriction to archimedean components
agrees with the restriction to archimedean components of the character
$e_F$. We then have the following adelic Fourier series expansion at 
infinity:\begin{align*}
f\left( \left( \begin{array} {cc} y& x \\ 0 & 1 \end{array} \right)\right) &= 
\vert y \vert_{{\bf{A}}_F} \cdot \sum_{\xi \in F \atop \xi \gg 0}
a_{\xi y \mathcal{O}_f}({\bf{f}})\cdot e_F(\xi {\bf{i}} y_{\infty})\cdot \chi_F(\xi x). \end{align*}
Here, ${\bf{i}}$ denotes the $d$-tuple $(i, \ldots, i)$, and $y_{\infty}$
denotes the archimedean component of $y$. Hence, we could also take this to be 
our definition of normalized Fourier coefficients. A Hilbert modular form 
${\bf{f}} \in \mathcal{S}_2(\mathfrak{N})$ is said to be a {\it{normalized eigenform}} 
if it is a simultaneous eigenvector for all of the Hecke operators $T_v$ and $U_v$ with
$a_{\mathcal{O}_F}({\bf{f}})=1$. In this case, we write
\begin{itemize}
\item $T_v {\bf{f}} = a_v({\bf{f}}) \cdot {\bf{f}}$ for all $v
\nmid \mathfrak{N}.$
\item $U_v {\bf{f}} = \alpha_v({\bf{f}}) \cdot {\bf{f}}$ for all $v
\mid \mathfrak{N}.$\end{itemize} A normalized eigenform ${\bf{f}} \in 
\mathcal{S}_2(\mathfrak{N})$ is said to be a {\it{newform}} if there does 
not exist any other form ${\bf{g}} \in \mathcal{S}_2(\mathfrak{M})$ with 
$\mathfrak{M} \mid \mathfrak{N}$ and $\mathfrak{M} \supsetneq \mathfrak{N}$ 
such that $a_{\mathfrak{n}}({\bf{f}}) = a_{\mathfrak{n}}({\bf{g}})$ for all
integral ideals $\mathfrak{n}$ of $F$ prime to $\mathfrak{N}$.
An eigenform ${\bf{f}} \in \mathcal{S}_2(\mathfrak{N})$
is {\it{$\mathfrak{p}$-ordinary}} if its $T_{\mathfrak{p}}$-eigenvalue
$a_{\mathfrak{p}}({\bf{f}})$ is a $p$-adic unit with respect to any fixed
embedding $\overline{{\bf{Q}}} \rightarrow \overline{{\bf{Q}}}_p$. In this case, there
exists a $p$-adic unit root $\alpha_{\mathfrak{p}}({\bf{f}})$ to the Hecke
polynomial \begin{align}\label{char} x^2 - a_{\mathfrak{p}}({\bf{f}})x + q,\end{align}
where $q$ denotes the cardinality of the residue field at $\mathfrak{p}$.
Let us now fix an integral ideal $\mathfrak{N} \subset \mathcal{O}_F$
as in $(\ref{N})$, with underlying integral ideal $\mathfrak{N}_0 
\subset \mathcal{O}_F$. Fix a Hilbert modular eigenform 
${\bf{f}}_0 \in \mathcal{S}_2(\mathfrak{N}_0)$ that is new at all primes
dividing the level $\mathfrak{N}_0$. Let $\mathfrak{N}_0 \subset \mathcal{O}_F$ 
be an integral ideal that is not divisible by $\mathfrak{p}$. 
The {\it{$\mathfrak{p}$-stabilization}}
${\bf{f}} \in \mathcal{S}_2(\mathfrak{N})$ of ${\bf{f}}_0 \in 
\mathcal{S}_2(\mathfrak{N}_0)$ is the
eigenform given by \begin{align} \label{p-stab} {\bf{f}} =
{\bf{f}}_0 - \beta_{\mathfrak{p}}({\bf{f}}_0)\cdot \left(
T_{\mathfrak{p}}{\bf{f}}_0\right),
\end{align} where $\beta_{\mathfrak{p}}({\bf{f}}_0)$ denotes the
non-unit root to $(\ref{char})$. This is a $\mathfrak{p}$-ordinary
eigenform in $\mathcal{S}_2(\mathfrak{N})$ with
$U_{\mathfrak{p}}$-eigenvalue $\alpha_{\mathfrak{p}} 
= \alpha_{\mathfrak{p}}({\bf{f}}_0)$. \end{remark}

Let us now consider an eigenform ${\bf{f}}
\in \mathcal{S}_2(\mathfrak{N})$ given by ${\bf{f}}_0$ if 
$\mathfrak{p}$ divides $\mathfrak{N}_0$, or given by the 
$\mathfrak{p}$-stabilization of ${\bf{f}}_0$ if $\mathfrak{p}$ 
does not divide $\mathfrak{N}_0$. We have the following 
quaternionic description of ${\bf{f}}$ in either case.
Let $B$ denote the totally definite quaternion algebra over $F$ of
discriminant $\mathfrak{N}^{-}$. To be consistent with the notations above, 
let us also write $U_v$ for the Hecke operators $T_v$ on $S_2^B(H; \mathcal{O})$ 
when $v \mid \mathfrak{N}^+$.

\begin{proposition}[Jacquet-Langlands]
\label{p-JLC} Given an eigenform ${\bf{f}} \in
\mathcal{S}_2(\mathfrak{N})$ as above, there exists a
function $\Phi \in \mathcal{S}_2\left( \coprod_{i=1}^{h}
\Gamma_i \backslash \mathcal{T}_{\mathfrak{p}}; {\bf{C}} \right)$
such that

\begin{itemize}
\item $T_v \Phi = a_v({\bf{f}}) \cdot \Phi$ for all $v
\nmid \mathfrak{N}$
\item $U_v \Phi = \alpha_v({\bf{f}}) \cdot \Phi$ for all $v
\mid \mathfrak{N}^{+}$
\item $U_{\mathfrak{p}} \Phi = \alpha_{\mathfrak{p}} \cdot \Phi.$
\end{itemize} This function is unique up to multiplication by
non-zero complex numbers. Conversely, given an eigenform $\Phi \in
\mathcal{S}_2\left( \coprod_{i=1}^{h} \Gamma_i \backslash
\mathcal{T}_{\mathfrak{p}}; {\bf{C}} \right)$, there exists an
eigenform ${\bf{f}} \in \mathcal{S}_2(\mathfrak{N})$ such that

\begin{itemize}
\item $T_v {\bf{f}} = a_v(\Phi) \cdot {\bf{f}}$ for all $v
\nmid \mathfrak{N}$
\item $U_v {\bf{f}} = \alpha_v(\Phi) \cdot {\bf{f}}$ for all $v
\mid \mathfrak{N}^{+}$
\item $U_{\mathfrak{p}} {\bf{f}} = \alpha_{\mathfrak{p}}(\Phi) \cdot {\bf{f}}.$
\end{itemize} Here, $a_v(\Phi)$ denotes the eigenvalue for $T_v$ of
$\Phi$ if $v \nmid \mathfrak{N}$, and $\alpha_v(\Phi)$ the
eigenvalue for $U_v$ of $\Phi$ if $v \mid \mathfrak{N}$.
\end{proposition}

\begin{proof} We generalize the argument given for
$F={\bf{Q}}$ in \cite[Proposition 1.3]{BD}.

\begin{remark}[Case I:] Suppose first that $\mathfrak{p}$
divides $\mathfrak{N}_0 = \mathfrak{N}$, hence that 
${\bf{f}} \in \mathcal{S}_2(\mathfrak{N})$ is new at
$\mathfrak{p}$. Let $R_0 \subset B$ be an Eichler order of level
$\mathfrak{p} \mathfrak{N}^{+}$. The theorem of Jacquet-Langlands
\cite{JL} then associates to ${\bf{f}}$ an eigenform $\Phi \in
\mathcal{S}_2^B(\widehat{R}_0^{\times}; {\bf{C}})$ such that

\begin{itemize}
\item[] $T_v \Phi = a_v({\bf{f}}) \cdot \Phi$ for all $v
\nmid \mathfrak{N}$
\item[] $U_v \Phi = \alpha_v({\bf{f}}) \cdot \Phi$ for all $v
\mid \mathfrak{N}^{+}$
\item[] $U_{\mathfrak{p}} \Phi = \alpha_{\mathfrak{p}} \cdot \Phi$.
\end{itemize} This function $\Phi$ is unique up to multiplication by
nonzero elements of ${\bf{C}}$. It can be identified with a function
in $\mathcal{S}_2\left( \coprod_{i=1}^{h} \Gamma_i
\backslash \mathcal{T}_{\mathfrak{p}}; {\bf{C}}\right)$ by
Proposition
\ref{SA3}.\end{remark}

\begin{remark}[Case II:] Suppose that $\mathfrak{p}$ does not divide
$\mathfrak{N}_0$, hence that ${\bf{f}} \in \mathcal{S}_2(\mathfrak{N})$ 
is not new at $\mathfrak{p}$.  Let $R_0 \subset B$ be an Eichler order 
of level $\mathfrak{N}^{+}$. The theorem of Jacquet-Langlands \cite{JL} 
then associates to ${\bf{f}}$ an eigenform $\Phi_0 \in \mathcal{S}_2^B
(\widehat{R}_0^{\times}; {\bf{C}})$ such that

\begin{itemize}
\item[] $T_v \Phi_0 = a_v({\bf{f}}) \cdot \Phi_0$ for all $v \nmid \mathfrak{N}_0$.
\item[] $U_v \Phi_0 = \alpha_v({\bf{f}}) \cdot \Phi_0$ for all $v
\mid \mathfrak{N}^{+}.$ \end{itemize} This function $\Phi_0$ is also unique 
up to multiplication by nonzero elements of ${\bf{C}}$. It can be identified with 
a function in $\mathcal{S}_2\left( \coprod_{i=1}^{h} \Gamma_i \backslash 
\mathcal{V}_{\mathfrak{p}}; {\bf{C}} \right)$ by Proposition \ref{SA3}. Now, we 
can construct from this function $\Phi_0 = \left(\phi_0^i \right)_{i=1}^{h}$ the 
following functions of $\left( {\phi}^{i} \right)_{i=1}^h$ of $\mathcal{S}_2\left( 
\coprod_{i=1}^{h} \Gamma_i \backslash \mathcal{E}_{\mathfrak{p}}^{*}; 
{\bf{C}}\right).$ For each component $\phi_0^i$, define a pair of functions 
$\phi^i_s, \phi^{i}_t: \mathcal{E}_{\mathfrak{p}}^{*} \longrightarrow {\bf{C}}$ by the
rules \begin{align}\label{rules} \phi_{s}^{i}(\mathfrak{e}) =
\phi_{0}^{i}(s(\mathfrak{e})), ~~~~~ \phi_{t}^{i}(\mathfrak{e}) =
\phi_{0}^{i}(t(\mathfrak{e})). \end{align} We have by construction
that \begin{itemize}
\item[] $T_v \phi_s^i = a_v({\bf{f}}) \cdot \phi_s^i$ for all $v
\nmid \mathfrak{N}_0$.
\item[] $T_v \phi_t^i = a_v({\bf{f}}) \cdot \phi_t^i$ for all $v
\nmid \mathfrak{N}_0$.\end{itemize} Now, observe that
\begin{align}\label{r1}
\left( U_{\mathfrak{p}} \phi_s^i\right) (\mathfrak{e}) =
\sum_{s(\mathfrak{e}') = t(\mathfrak{e})}\phi_s^i(\mathfrak{e}') =
\sum_{s(\mathfrak{e}') = t(\mathfrak{e})}\phi_0^i(t(\mathfrak{e})) =
q \cdot \phi_t^i(\mathfrak{e}).
\end{align} On the other hand, observe that \begin{align}\label{r2}
\left( U_{\mathfrak{p}} \phi_t^i\right) (\mathfrak{e}) =
\sum_{s(\mathfrak{e}') = t(\mathfrak{e})} \phi_t^i(\mathfrak{e}) =
\left( T_{\mathfrak{p}} \phi_0^i \right) (t(\mathfrak{e})) -
\phi_0^i (s(\mathfrak{e})) = a_{\mathfrak{p}}({\bf{f}}_0) \cdot
\phi_t^i(\mathfrak{e}) - \phi_s^i(\mathfrak{e}).
\end{align} Let us
now define $\phi^i = \phi_s^i - \alpha_{\mathfrak{p}} \cdot
\phi_t^i. $ Using $(\ref{r1})$ and $(\ref{r2})$, we find that

\begin{align*} \left(U_{\mathfrak{p}} \phi^i \right)(\mathfrak{e})
&= \sum_{ s(\mathfrak{e}') = t(\mathfrak{e}) }
\phi_s^i(\mathfrak{e}') - \alpha_{\mathfrak{p}}
\cdot \phi_t^i(\mathfrak{e}')  \\
&= \sum_{s(\mathfrak{e}') = t(\mathfrak{e})}
\phi_0^i(s(\mathfrak{e}')) - \alpha_{\mathfrak{p}} \cdot
\sum_{s(\mathfrak{e}') = t(\mathfrak{e})} \phi_0^i(t(\mathfrak{e}'))
\\ &= q \cdot \phi_0^i(t(\mathfrak{e})) - \alpha_{\mathfrak{p}}
\left( a_{\mathfrak{p}}({\bf{f}}_0) \cdot \phi_0^i(t(\mathfrak{e}))
- \phi_0^{i}(s(\mathfrak{e})) \right) \\
&= \alpha_{\mathfrak{p}} \cdot \phi_s^i(\mathfrak{e}) - \left(
\alpha_{\mathfrak{p}} \cdot a_{\mathfrak{p}}({\bf{f}}_0) - q \right)
\cdot \phi_t^i(\mathfrak{e}) \\ &= \alpha_{\mathfrak{p}} \cdot
\left( \phi_s^i(\mathfrak{e}) - \alpha_{\mathfrak{p}} \cdot
\phi_t^i(\mathfrak{e}) \right) \\
&= \alpha_{\mathfrak{p}} \cdot \phi^i(\mathfrak{e}),\end{align*}
i.e. since $\alpha_{\mathfrak{p}}$ being a root of $(\ref{char})$
implies that $\alpha_{\mathfrak{p}}^2 = a_{\mathfrak{p}}({\bf{f}}_0)
\cdot \alpha_{\mathfrak{p}} - q$. It follows that $\phi^i$ is an
eigenvector for the Hecke operator $U_{\mathfrak{p}}$ with
eigenvalue $\alpha_{\mathfrak{p}}$. Let us then define $\Phi =
\left( \phi^i \right)_{i=1}^{h}$. Thus, $\Phi$ is a
function in the space $\mathcal{S}_2\left(
\coprod_{i=1}^{h} \Gamma_i \backslash
\mathcal{E}_{\mathfrak{p}}^{*} ; {\bf{C}} \right)$. It is an simultaneous
eigenvector for all of the operators $T_v$ with $v \nmid
\mathfrak{N}$, $U_v$ with $v \mid \mathfrak{N}^+$, and
$U_{\mathfrak{p}}$ having the prescribed eigenvalues. Moreover, it
is the unique such function up to multiplication by non-zero
elements of ${\bf{C}}$.\end{remark}

The converse in either case can be established as follows. Given
such a function $\Phi$ in $\mathcal{S}_2\left(
\coprod_{i=1}^{h} \Gamma_i \backslash
\mathcal{T}_{\mathfrak{p}} ; {\bf{C}} \right)$, consider its image
under the bijection $(\ref{bij})$. The theorem
of Jacquet and Langlands then associates to this image an
eigenform ${\bf{f}} \in \mathcal{S}_2(\mathfrak{N})$ having the
proscribed eigenvalues. \end{proof} \end{remark}

\section{Construction of measures}

Let ${\bf{f}} \in \mathcal{S}_2(\mathfrak{N})$
be an eigenfrom as defined for Proposition \ref{p-JLC}, with $\Phi$ the
associated quaternionic eigenform. Recall that we fixed an embedding 
$\overline{\bf{Q}} \rightarrow \overline{\bf{Q}}_p$, and take
$\mathcal{O}$ to be the ring of integers of a finite extension of ${\bf{Q}}_p$
containing all of the Fourier coefficients of ${\bf{f}}$. Recall as well that we let $\Lambda$
denote the $\mathcal{O}$-Iwasawa algebra $\mathcal{O}[[G_{\mathfrak{p}^{\infty}}]]$. We
construct in this section elements of $ \Lambda$, equivalently $\mathcal{O}$-valued 
measures on $G_{\mathfrak{p}^{\infty}}$, that 
interpolate the central values $L(\Phi, \rho, 1/2)$. Here, $\rho$ is any 
finite order character of $G_{\mathfrak{p}^{\infty}}$. The construction 
below generalizes those of Bertolini-Darmon (\cite{BD2}, \cite{BD}) in the 
ordinary case, as well as constructions of Darmon-Iovita \cite{DI} and Pollack
\cite{Poll} in the supersingular case. These constructions are also
sketched in the ordinary case by Longo \cite{Lo2}, using the language of Gross points. 
We use the Yuan-Zhang-Zhang generalization Waldpurger's theorem, as described in
Theorem \ref{W} above, to deduce an interpolation formula
for these measures (Theorem \ref{basicint}). We then give
a formula for the associated $\mu$-invariant (Theorem \ref{mu}),
generalizing the work of Vatsal \cite{Va2}. 

Fix an integral ideal $\mathfrak{N} \subset \mathcal{O}_F$ having the 
factorization in $F$ defined in $(\ref{factorization})$.
Let $B$ denote the totally definite quaternion algebra over
$F$ of discriminant $\mathfrak{N}^{-}$. Let $Z$ denote the 
maximal $\mathcal{O}_F[\frac{1}{\mathfrak{p}}]$-order in $K$, and let
$R \subset B$ be an Eichler $\mathcal{O}_F[\frac{1}{\mathfrak{p}}]$-order 
of level $\mathfrak{N}^{+}$. We fix an {\it{optimal embedding}} $\Psi$ of $\mathcal{O}$ 
into $R$, i.e. an injective $F$-algebra homomorphism $ \Psi: K \longrightarrow B$ 
such that \begin{align*}\Psi(K) \cap R = \Psi(Z).\end{align*} 
Such an embedding exists if and only if all of the primes dividing 
the level of $R$ are split in $K$ (see \cite[$\S$ II.3]{Vi}), so 
our choice of factorization $(\ref{factorization})$ ensures that we 
may choose such an embedding.  

\begin{remark}[Galois action on the Bruhat-Tits tree.] 

The reciprocity map $\rec_K$ induces an isomorphism \begin{align*}\begin{CD} 
\widehat{K}^{\times}/ \left(K^{\times} \prod_{v \nmid \mathfrak{p}} 
Z_{v}^{\times}\right) @>{r_K}>> G[\mathfrak{p}^{\infty}],\end{CD}\end{align*} 
as implied for instance from Lemma \ref{Galois}(i). Passing to the adelization, 
the optimal embedding $\Psi$ then induces an embedding \begin{align*}\begin{CD} 
\widehat{K}^{\times}/ \left(K^{\times} \prod_{v \nmid \mathfrak{p}}
Z_{v}^{\times}\right) @>{\widehat{\Psi}}>> B^{\times}\backslash \widehat{B}^{\times} / 
\prod_{v\nmid \mathfrak{p}} R_{v}^{\times}. \end{CD}\end{align*} 
Taking the subset of $\widehat{B}^{\times}$ defined by 
$\prod_{v \nmid \mathfrak{p}}R_{v}^{\times}$, with associated subgroups 
$\Gamma_{i}$ as defined in $(\ref{gamma})$, strong approximation $(\ref{SA2})$ 
gives an isomorphism \begin{align*}\begin{CD} \coprod_{i=1}^{h} \Gamma_{i}
\backslash B_{\mathfrak{p}}^{\times}/ F_{\mathfrak{p}}^{\times}
@>{\eta_{\mathfrak{p}}}>> B^{\times} \backslash \widehat{B}^{\times}
/  \prod_{v \nmid \mathfrak{p}} R_{v}^{\times}.
\end{CD}\end{align*} We may therefore consider the composition 
of maps given by \begin{align*}\begin{CD} 
G[\mathfrak{p}^{\infty}] @>{r_K^{-1}}>>
\widehat{K}^{\times}/ \left( K^{\times} \prod_{v
\nmid \mathfrak{p}} Z_{v}^{\times}\right)  \\
@>{\widehat{\Psi} }>> B^{\times} \backslash \widehat{B}^{\times} /
\prod_{v \nmid \mathfrak{p}} R_{v}^{\times}
 \\ @>{\eta_{\mathfrak{p}}^{-1}}>> \coprod_{i=1}^{h} \Gamma_{i}
\backslash B_{\mathfrak{p}}^{\times}/ F_{\mathfrak{p}}^{\times}.
\end{CD}\end{align*} The composition 
\begin{align}\label{compos} 
\eta_{\mathfrak{p}}^{-1} \circ \widehat{\Psi} \circ
r_{K}^{-1}\end{align} induces an action $\star$ of
the Galois group $G[\mathfrak{p}^{\infty}]$ on the Bruhat-Tits tree
$\mathcal{T}_{\mathfrak{p}} = (\mathcal{V}_{\mathfrak{p}},
\mathcal{E}_{\mathfrak{p}}^*)$. This action factors through that of the 
local optimal embedding $\Psi_{\mathfrak{p}}: K_{\mathfrak{p}} \longrightarrow B_{\mathfrak{p}}$.
We can give a more precise description, following \cite[$\S$ 2.2]{DI}. That is, 
the optimal embedding $\Psi: K \longrightarrow B$ induces a local optimal embedding 
$\Psi_{\mathfrak{p}}: K_{\mathfrak{p}}^{\times} \longrightarrow
B_{\mathfrak{p}}^{\times}$, which in turn induces an action (by conjugation) 
of $K_{\mathfrak{p}}^{\times}/F_{\mathfrak{p}}^{\times}$ on
$B_{\mathfrak{p}}^{\times}/F_{\mathfrak{p}}^{\times}$. The dynamics of this action depend on 
the decomposition of $\mathfrak{p}$ in $K$. Hence, we divide into cases. Let us write 
$K_{\mathfrak{p}} = K \otimes F_{\mathfrak{p}}$ to denote the localization of $K$ at $\mathfrak{p}$.

\begin{remark} [Case I: $\mathfrak{p}$ splits in $K$.] In this case, $K_{\mathfrak{p}}^{\times}$
is a split torus, and so the action of $K_{\mathfrak{p}}^{\times}/F_{\mathfrak{p}}^{\times}$ on
$\mathcal{T}_{\mathfrak{p}}$ does not fix any vertex. Fix a prime
$\mathfrak{P}$ above $\mathfrak{p}$ in $K$ (not to be confused with the maximal ideal 
$\mathfrak{P} \subset \mathcal{O}$ defined above). Define a homomorphism \begin{align*} \vert\vert
~\vert\vert_{\mathfrak{P}} : K_{\mathfrak{p}}^{\times}/F_{\mathfrak{p}}^{\times} &\longrightarrow {\bf{Z}}\\
x &\longmapsto \ord_{\mathfrak{P}}\left( \frac{x}{\overline{x}} \right).
\end{align*} Note that the choice of $\mathfrak{P}$ only changes the homomorphism above
by a sign. For later applications, we shall choose $\mathfrak{P}$ in accordance with our fixed 
embedding $\overline{{\bf{Q}}} \rightarrow \overline{{\bf{Q}}}_p$. Consider the maximal compact 
subgroup of $K_{\mathfrak{p}}^{\times}/F_{\mathfrak{p}}^{\times}$ defined by
\begin{align*} \UU_0 = \ker \left( \vert\vert ~ \vert\vert_{\mathfrak{P}}\right).
\end{align*} Consider the natural decreasing filtration of compact
subgroups \begin{align}\label{filtration1} \ldots \subset \UU_j \subset \ldots \UU_1
\subset \UU_0\end{align} satisfying the condition that $[\UU_0: \UU_j] = q^{j-1}(q +1)$
for each index $j \geq 1$. The action of $\UU_0$ on $\mathcal{T}_{\mathfrak{p}}$ fixes a
{\it{geodesic}} $J = J_{\Psi}$ of $\mathcal{T}_{\mathfrak{p}}$, i.e.
an infinite sequence of consecutive vertices. Now, the quotient
$K_{\mathfrak{p}}^{\times}/F_{\mathfrak{p}}^{\times}/\UU_0$ acts by translation on $J$. 
Let us define the distance between any vertex $\mathfrak{v} \in
\mathcal{V}_{\mathfrak{p}}$ and the geodesic $J$ to be
\begin{align*} d(\mathfrak{v}, J) := \min_{\mathfrak{w} \in J}
d(\mathfrak{v}, \mathfrak{w}).\end{align*} Here, $\mathfrak{w} \in
J$ runs over all of the vertices of $J$. If $d(\mathfrak{v}, J) =
j$, then it is simple to see from the definitions of distance above
that $\Stab_{K_{\mathfrak{p}}^{\times}/F_{\mathfrak{p}}^{\times}}\left( \mathfrak{v} 
\right) = \UU_j$. Moreover, we see that the quotient $K_{\mathfrak{p}}^{\times}/
F_{\mathfrak{p}}^{\times}/\UU_j$ acts simply transitively on the set of vertices 
of distance $j$ from $J$. In this case, let us fix a sequence of consecutive 
vertices $\lbrace \mathfrak{v}_j \rbrace_{j \geq 0}$ with $d(\mathfrak{v}_j, J) = j$
such that $\Stab_{K_{\mathfrak{p}}^{\times}/F_{\mathfrak{p}}^{\times}}\left( \mathfrak{v}_j \right) =
\UU_j.$\end{remark}

\begin{remark}[Case II: $\mathfrak{p}$ is inert in $K$.]
In this case, the quotient $K_{\mathfrak{p}}^{\times}/F_{\mathfrak{p}}^{\times}$ is compact, 
and so the action of $K_{\mathfrak{p}}^{\times}/F_{\mathfrak{p}}^{\times}$ on
$\mathcal{T}_{\mathfrak{p}}$ fixes a distinguished vertex
$\mathfrak{v}_0$. Hence, we can take $\UU_0 =
K_{\mathfrak{p}}^{\times}/F_{\mathfrak{p}}^{\times}$ to be the maximal 
compact subgroup in the construction above, with associated natural decreasing filtration
of compact subgroups \begin{align}\label{filtration2} \ldots \subset \UU_j \subset \ldots \UU_1 \subset \UU_0 =
K_{\mathfrak{p}}^{\times}/F_{\mathfrak{p}}^{\times}\end{align} satisfying the condition that 
$[\UU_0: \UU_j] = q^{j-1}(q +1)$ for each index $j\geq 1$. If $d(\mathfrak{v}_0, \mathfrak{v}) = j$ 
for some vertex $\mathfrak{v} \in \mathcal{V}_{\mathfrak{p}}$, then we have that
$\Stab_{K_{\mathfrak{p}}^{\times}/F_{\mathfrak{p}}^{\times}}\left(\mathfrak{v}\right) = \UU_j.$ 
In this case, let us fix a sequence of consecutive vertices $\lbrace
\mathfrak{v}_j \rbrace_{j\geq 0}$ with $d(\mathfrak{v}_0,
\mathfrak{v}_j)=j$ such that $\Stab_{K_{\mathfrak{p}}^{\times}/F_{\mathfrak{p}}^{\times}}\left(
\mathfrak{v}_j \right) = \UU_j.$\end{remark} Let us note that in either case
of the decomposition of $\mathfrak{p}$ in $K$, the the filtration subroup $\UU_j$ 
is simply the standard compact subgroup of $K_{\mathfrak{p}}^{\times}/F_{\mathfrak{p}}^{\times}$ 
of the form \begin{align}\label{UUj} \UU_j = \left( 1 + \mathfrak{p}^j \mathcal{O}_{K} \otimes
\mathcal{O}_{F_{\mathfrak{p}}}\right)^{\times}/\left( 1 +
\mathfrak{p}^j \mathcal{O}_{F_{\mathfrak{p}}}\right)^{\times}.\end{align}
In either case, we obtain from the filtration $(\ref{filtration1})$ or $(\ref{filtration2})$
an infinite sequence of consecutive edges $\lbrace \mathfrak{e}_j \rbrace_{j\geq 1}$, 
with each edge $\mathfrak{e}_j$ joining two vertices $\mathfrak{v}_{j-1} \leftrightarrow \mathfrak{v}_j$, 
and satisfying the property that \begin{align*} \Stab_{K_{\mathfrak{p}}^{\times}/F_{\mathfrak{p}}^{\times}}\left(
\mathfrak{e}_j \right) = \UU_j.\end{align*} We refer the reader to \cite[$\S$ 1]{DI} for some more details.
Let us for simplicity write $\lbrace \mathfrak{w}_j\rbrace = \lbrace \mathfrak{w}_j \rbrace_j$ to 
denote either the sequence of consecutive vertices $\lbrace 
\mathfrak{v}_j \rbrace_{j \geq 0}$ or the induced sequence of consecutive 
edges $\lbrace \mathfrak{e}_j \rbrace_{j\geq1}.$ 
\end{remark}

\begin{remark}[A pairing.]

Fix an eigenform $\Phi = \left( \phi^i \right)_{i=1}^{h} \in \mathcal{S}_{2}(\coprod_{i=1}^{h} 
\Gamma_i \backslash \mathcal{T}_{\mathfrak{p}}; \mathcal{O})$. Fix a sequence of consecutive
edges or vertices $\lbrace \mathfrak{w}_j \rbrace$. We define for each index $j$ a function
\begin{align*}\Phi_{K, j}: K_{\mathfrak{p}}^{\times}/F_{\mathfrak{p}}^{\times}/\UU_j 
&\longrightarrow \mathcal{O},\\ \gamma &\longmapsto \Phi\left( 
\gamma \star \mathfrak{w}_j \right) = \left( \phi^i \left( 
\gamma \star \mathfrak{w}_{j} \right) \right)_{i=1}^{h}.\end{align*} 
Let us now simplify notation by writing 
\begin{align*} \HH_{\infty} &= \rec_K^{-1}\left(G[\mathfrak{p}^{\infty}]\right) =
\widehat{K}^{\times}/ \left( K^{\times}
\prod_{v \nmid \mathfrak{p}} Z_{v}^{\times}\right).
\end{align*} Let us commit an abuse of notation in writing $\UU_j$ 
to also denote the image of the filtration subgroup defined in $(\ref{UUj})$above in 
$\HH_{\infty}$. We then have the relation \begin{align*} \HH_{\infty} = \ilim j \HH_{\infty}/\UU_j.
\end{align*} To be more precise, $\HH_{\infty}$ is profinite, hence compact. The 
open subgroups $\UU_j$ then have finite index in $\HH_{\infty}$. Since 
$\HH_{\infty}$ must also be locally compact, its open subgroups form 
a base of neighbourhoods of the identity. We claim that the collection 
$\lbrace \UU_j \rbrace_{j \geq 0}$ in fact forms a base of neighbourhoods 
of the identity, in which case the natural map $\HH_{\infty} \longrightarrow 
\varprojlim_j \HH_{\infty}/\UU_j$ is seen to be both continuous and injective. 
Since its image is dense, a standard compactness argument then implies 
that the map must be an isomorphism. We now claim that the functions 
$\Phi_{K,j}$ defined above in fact descend to functions on the quotients 
$\HH_j = \HH_{\infty}/\UU_j$. Indeed, this we claim is clear
from the composition of maps $(\ref{compos})$, as the part of the
image of $\HH_{\infty}$ that does not land in $K_{\mathfrak{p}}^{\times}/
F_{\mathfrak{p}}^{\times}$ must  lie in one of the subgroups $\Gamma_i$. 
Since each eigenform ${\phi}^i $ is $\Gamma_i$-invariant,
the claim follows. The functions $\Phi_{K, j}$ are then seen to give 
rise to a natural pairing \begin{align*} [~,~]_{\Phi}:
\HH_{\infty} \times \mathcal{T}_{\mathfrak{p}} &\longrightarrow 
\mathcal{O} \\ (  t, \mathfrak{w}_j) &\longmapsto \Phi(
\eta_{\mathfrak{p}}^{-1} \circ \widehat{\Psi}( t) \star
\mathfrak{w}_j),\end{align*} and under the reciprocity map $\rec_{K}$
a natural pairing \begin{align*} [~,~]_{\Phi}: G[\mathfrak{p}^{\infty}] \times
\mathcal{T}_{\mathfrak{p}} &\longrightarrow \mathcal{O}
\\ (\sigma, \mathfrak{w}_j) &\longmapsto \Phi( \eta_{\mathfrak{p}}^{-1} 
\circ \widehat{\Psi} \circ \rec_K^{-1} (\sigma)\star
\mathfrak{w}_j).\end{align*} Let us write $[~,~]_{\Phi}$ to denote either pairing,
though it is a minor abuse of notation.\end{remark}

\begin{remark} [A distribution.] 

Let us define for each index $j$ a group ring element 
\begin{align*} \vartheta_{\Phi, j} = \sum_{t \in \HH_j} \Phi_{K, j}(t) 
\cdot t \in \mathcal{O}[\HH_j]. \end{align*} We shall consider the 
natural projections \begin{align*}\pi_{j+1, j}: \HH_{j+1} \longrightarrow 
\HH_j,\end{align*} as well as the group ring operations 
\begin{align*} \xi_j = \sum_{x \in \UU_j/\UU_{j+1}} x.\end{align*}
That is, let $\xi_j$ denote the map from $\HH_j$ to 
$\HH_{j+1}$ such that for for any $y \in \HH_j$,
\begin{align} \label{xi} \xi_j(y) &= \sum_{x \in \HH_{j+1} \atop
\pi_{j+1, j}(x)=y} x. \end{align} Let us now clarify some more
notations. Given any $t_j \in \HH_j$, we let $t_{j+1}$ denote
an arbitrary lift of $t_j$ to $\HH_{j+1}$. This allows us to
abuse notation in viewing $\vartheta_{\Phi, j}$ as an element of the
group ring $\mathcal{O}[\HH_{j+1}],$ i.e. via replacement of $\vartheta_{\Phi,
j}$ with some arbitrary lift to $\mathcal{O}[\HH_{j+1}]$ under the projection
$\pi_{j+1, j}$. This lift is not well defined, but the product
$\xi_j \vartheta_{\Phi,j} \in \mathcal{O}[\HH_{j+1}]$ is.

\begin{lemma}\label{distribution} We have the following distribution
relations with respect to the eigenform ${\bf{f}}_0$ associated to
$\Phi$ in the setting of Proposition \ref{p-JLC}.

\begin{itemize}
\item[(i)] If $\lbrace \mathfrak{w}_j \rbrace =\lbrace
\mathfrak{v}_j \rbrace_{j\geq0}$, then
\begin{align} \pi_{j+1, j} \left( \vartheta_{\Phi, j+1} \right) =
a_{\mathfrak{p}}({\bf{f}}_0) \cdot \vartheta_{\Phi, j} - \xi_{j-1}
\vartheta_{\Phi, j-1}. \end{align}

\item[(ii)] If $\lbrace \mathfrak{w}_j \rbrace = \lbrace
\mathfrak{e}_j \rbrace_{j\geq1}$, then
\begin{align} \pi_{j+1, j} \left( \vartheta_{\Phi, j+1} \right) =
\alpha_{\mathfrak{p}}({\bf{f}}_0) \cdot \vartheta_{\Phi, j}.
\end{align} \end{itemize} \end{lemma}

\begin{proof}
See \cite[Lemma 2.6]{DI}. The same approach applied to each component $\phi^i$ 
of $\Phi$ works here. That is, we have by direct calculation on each $\phi^i$ that 
\begin{align}\label{projection} \pi_{j+1, j}\left( \vartheta_{\Phi, j+1} \right) &= 
\sum_{t_j \in \HH_j} \left( \sum_{x \in \HH_{j+1} \atop \pi_{j+1, j}(x)=y } \Phi_{K, j+1}
(x t_{j+1} ) \right) \cdot t_{j}^{-1}. \end{align} On the other hand, we have by definition that
\begin{align}\label{innersum} \sum_{x \in \HH_{j+1} \atop \pi_{j+1, j}(x)=y} \Phi_{K,j+1} 
(x t_{j+1} ) &= \sum_{x \in \HH_{j+1} \atop \pi_{j+1, j}(x)=y} \Phi\left((x t_{j+1}) \star 
\mathfrak{w}_{j+1}\right).\end{align} Suppose that $\lbrace \mathfrak{w}_j \rbrace 
= \lbrace\mathfrak{v}_j \rbrace_{j\geq0}$ is a sequence of consecutive vertices. 
Then, the sum on the right hand side of $(\ref{innersum})$ corresponds on each 
component $\phi^i$ to the sum over the $q+1$ vertices adjacent to $t_j \star \mathfrak{v}_j$, 
minus the vertex $t_j \star \mathfrak{v}_{j-1}$. We refer the reader to \cite[Figure 3, p. 12]{DI} 
for a visual aid, as it also depicts the situation here. In particular, we find that the inner sum 
of $(\ref{projection})$ is given by $$T_{\mathfrak{p}} \left( \Phi \right)(t_j \star \mathfrak{v}_j) 
- \Phi\left(t_j \star \mathfrak{v}_{j-1} \right).$$ We may then deduce from Theorem
\ref{p-JLC} applied to $\Phi$ that \begin{align*} \sum_{x \in \HH_{j+1} \atop \pi_{j+1, j}(x)=y} 
\Phi\left( (x t_{j+1}) \star \mathfrak{v}_{j+1} \right) &= a_{\mathfrak{p}}({\bf{f}}_0) \cdot 
\Phi(t_j \star \mathfrak{v}_j ) - \Phi(t_j \star \mathfrak{v}_{j-1})\\ &= a_{\mathfrak{p}}({\bf{f}}_0) 
\cdot \Phi_{K, j}(t_j) - \Phi_{K, j-1}(t_j). \end{align*} The first part of the claim then
follows from $(\ref{projection})$ and $(\ref{innersum})$, using the
definition of $\vartheta_{\Phi, j}$. Suppose now that $\lbrace
\mathfrak{w}_j \rbrace = \lbrace \mathfrak{e}_j \rbrace_{j\geq1}$ is a 
consecutive sequence of edges. Then, the sum on the right hand side 
of $(\ref{innersum})$ corresponds on each component $\phi^i$ to the 
sum over the $q+1$ edges emanating from $t_j \star \mathfrak{e}_j$, 
minus the edge obtained by reversing orientation. In particular, we find 
that the inner sum of $(\ref{projection})$ is given by
$$U_{\mathfrak{p}} \left( \Phi \right)(t_j \star \mathfrak{e}_j ).$$ We may 
then deduce from Proposition \ref{p-JLC} that \begin{align*} \sum_{x \in \HH_{j+1} 
\atop \pi_{j+1, j}(x)=y} \Phi\left( (x t_{j+1} ) \star \mathfrak{e}_{j+1} \right) =
\alpha_{\mathfrak{p}}({\bf{f}}_0) \cdot \Phi\left((t_j \star \mathfrak{e}_j) \right).\end{align*} 
The second part of the claim then follows as before from $(\ref{projection})$ and
$(\ref{innersum})$, using the definition of $\vartheta_{\Phi, j}$. \end{proof}\end{remark}

\begin{remark}[The ordinary case.]

Let us assume now that the eigenform
${\bf{f}}_0$ is $\mathfrak{p}$-ordinary, i.e. that the image of
the eigenvalue $a_{\mathfrak{p}}({\bf{f}}_0)$ under our fixed embedding
$\overline{\bf{Q}} \rightarrow \overline{{\bf{Q}}}_p$ is a $p$-adic unit. 
Recall that we let $\alpha_{\mathfrak{p}} = \alpha_{\mathfrak{p}}({\bf{f}}_0)$ 
denote unit root of the Hecke polynomial $(\ref{char})$. Fix a sequence of 
consecutive directed edges $\lbrace \mathfrak{w}_j \rbrace = \lbrace
\mathfrak{e}_j \rbrace_{j\geq 1}$. Let us consider the system of maps
\begin{align*} \varphi_{\Phi, j}: \HH_j &\longrightarrow \mathcal{O} \end{align*}
defined for each index $j \geq 1$ by the assignment of an element $t \in \HH_j$
to the value \begin{align}\label{orddist} \varphi_{\Phi, j}(t) &= \alpha_{\mathfrak{p}}^{-j} 
\cdot \Phi_{K, j}(t). \end{align} For each $j \geq 1$, let us also define a group ring element 
\begin{align*}\theta_{\Phi, j} (\HH_j) &= \alpha_{\mathfrak{p}}^{-j} 
\cdot \sum_{t \in \HH_j} \Phi_{K,j}(t) \cdot  t \\ &= \alpha_{\mathfrak{p}}^{-j} 
\cdot \vartheta_{\Phi, j} \in \mathcal{O}[\HH_j]. \end{align*}

\begin{lemma} The system of maps $\lbrace \varphi_{\Phi, j}
\rbrace_{j \geq 1}$ defined in $(\ref{orddist})$ determines an $\mathcal{O}$-valued 
measure on the group $\HH_{\infty} = \rec_K^{-1}(G[\mathfrak{p}^{\infty}])$. \end{lemma}

\begin{proof} Lemma \ref{distribution}(ii) implies that for each $j
\geq 1$, \begin{align*} \pi_{j+1, j}\left( \theta_{\Phi, j+1}
\right) = \pi_{j+1,j } \left( \alpha_{\mathfrak{p}}^{-(j+1)} \cdot
\vartheta_{\Phi, j+1} \right)  &= \alpha_{\mathfrak{p}}^{-(j+1)}
\cdot \pi_{j+1, j} \left( \vartheta_{\Phi, j+1} \right) \\ &=
\alpha_{\mathfrak{p}}^{-(j+1)} \cdot \alpha_{\mathfrak{p}} \cdot
\vartheta_{\Phi, j} \\ &= \alpha_{\mathfrak{p}}^{-j} \cdot
\vartheta_{\Phi, j} \\ &= \theta_{\Phi, j}.\end{align*} Hence, the
system of maps $\lbrace \varphi_{\Phi, j} \rbrace_{j \geq 1}$
defines a bounded $\mathcal{O}$-valued distribution on $H_{\infty}$, as
required. \end{proof}  

\begin{corollary} The system of maps $\lbrace \varphi_{\Phi, j}
\rbrace_{j\geq 1}$, under composition with the reciprocity map 
$\rec_K$ followed by projection to the Iwasawa algebra 
$\Lambda = \mathcal{O}[[G_{\mathfrak{p}^{\infty}}]]$,
defines an $\mathcal{O}$-valued measure $d \vartheta_{\Phi}$ 
on the Galois group $G_{\mathfrak{p}^{\infty}}$. \end{corollary} 
Let us now consider the associated completed group ring element
\begin{align}\label{ordzeta} \theta_{\Phi} = \ilim j
\alpha_{\mathfrak{p}}^{-j} \cdot \sum_{\sigma \in G_{\mathfrak{p}^j}} 
[\sigma, \mathfrak{e}_j]_{\Phi} \cdot
\sigma \in \Lambda
\end{align} Observe that a different choice of sequence of directed edges
$\lbrace \mathfrak{e}_j \rbrace_{j\geq1}$ has the effect of
multiplying $\theta_{\Phi}$ by an automorphism of
$G_{\mathfrak{p}^{\infty}}$. To correct this, we let
$\mathcal{L}_{\Phi}^{*}$ denote the image of $\mathcal{L}_{\Phi}$
under the involution of $\mathcal{O}[[G_{\mathfrak{p}^{\infty}}]]$ that sends
$\sigma \mapsto \sigma^{-1} \in G_{\mathfrak{p}^{\infty}}$.

\begin{definition}
Let $\mathcal{L}_{\mathfrak{p}}(\Phi, K) = \theta_{\Phi}
\theta_{\Phi}^{*}.$ \end{definition} Hence, $\mathcal{L}_{\mathfrak{p}}(\Phi, K) $ is a well-defined 
element of $\Lambda$.\end{remark}

\begin{remark} [The supersingular case.]

Assume now that $a_{\mathfrak{p}}({\bf{f}}_0) = 0.$ Fix a sequence
of consecutive vertices $\lbrace \mathfrak{w}_j \rbrace = \lbrace 
\mathfrak{v}_j \rbrace_{j\geq 0}$. Here, we give a construction of 
the $\mathfrak{p}$-adic $L$-function of the quaternionic eigenform 
$\Phi$ associated to ${\bf{f}}_0$ by Proposition \ref{p-JLC} following 
\cite{DI}, building on techniques of Pollack \cite{Poll}. Recall that 
by Lemma \ref{Galois} (ii), we have an isomorphism of topological 
groups $G_{\mathfrak{p}^{\infty}} \cong {\bf{Z}}_{p}^{\delta}$, with 
$\delta = [F_{\mathfrak{p}} : {\bf{Q}}_p].$ Fixing $\delta$ topological 
generators $\gamma_1, \ldots, \gamma_{\delta}$ of $G_{\mathfrak{p}^{\infty}}$,
we can then define an isomorphism \begin{align*} \Lambda &\longrightarrow 
\mathcal{O}[[T_1, \ldots, T_{\delta}]] \\ \left( \gamma_1, \ldots, \gamma_{\delta} \right)
&\longmapsto \left( T_1 + 1, \ldots, T_{\delta} + 1 \right). \end{align*} 
We obtain from this an identification of group rings
\begin{align}\label{groupring} \mathcal{O}[G_{\mathfrak{p}^n}] \longrightarrow 
\mathcal{O}[T_1, \ldots, T_{\delta}]/\left( \left( T_1+1 \right)^{p^n} -1, \ldots, 
\left( T_{\delta}+1 \right)^{p^n}-1 \right)\end{align} via the map that sends each 
$\gamma_i \mod G_{\mathfrak{p}^{\infty}}^{p^n}$ to the class $T_i +1 \mod 
\left((T_i+1)^{p^n} -1\right).$ Granted this identification 
$(\ref{groupring})$, we claim to have the following power series description 
of the group ring operator $\xi_n$ defined in $(\ref{xi})$: 
\begin{align*}\begin{CD} (T_1, \ldots, T_{\delta}) @>{\xi_n}>> 
\left( \Sigma_{p^n}(T_1+1), \ldots, \Sigma_{p^n}(T_{\delta}+1) \right).\end{CD}\end{align*} 
Here, $\Sigma_{p^n}$ denotes the cyclotomic polynomial of degree $p^n$. Let $\Omega_n$ 
denote the power series operation that sends \begin{align*} (T_1, \ldots, T_{\delta}) 
&\longrightarrow  \left(  \left( T_1+1 \right)^{p^n} -1, \ldots, \left(T_{\delta} +1\right)^{p^n}-1 \right) \\
&=  \left( T_1 \cdot \prod_{j=1}^{n} \Sigma_{p^j}(T_1 + 1), \ldots, T_{\delta} \cdot \prod_{j=1}^{n}
\Sigma_{p^j}(T_{\delta}+1) \right). \end{align*} Here, the last equality follows from the 
fact that \begin{align*} \left(T+1 \right)^{p^n} = T \cdot \prod_{j=1}^n \Sigma_{p^j} (T+1).
\end{align*} Let us also define power series operations
\begin{align*} \widetilde{\Omega}_n^{+} &=
\widetilde{\Omega}_n^{+}(T_1, \ldots, T_{\delta}) = \prod_{j=2 \atop j
\equiv 0(2)}^n \xi_j(T_1, \ldots, T_{\delta})\\
\widetilde{\Omega}_n^{-} &= \widetilde{\Omega}_n^{-}(T_1, \ldots,
T_{\delta})  = \prod_{j=1 \atop j \equiv 1(2)}^n \xi_j(T_1, \ldots, T_{\delta})\\
\Omega_{n}^{\pm} &= \Omega_{n}^{\pm}(T_1, \ldots, T_{\delta})  = (T_1,
\ldots, T_{\delta}) \star \widetilde{\Omega}_{n}^{\pm}. \end{align*} 
Here, we write $( T_1, \ldots, T_{\delta}) \star $ to denote the dot product, i.e. the multiplication
operation that sends $(X_1, \ldots, X_{\delta})$ to $ (T_1 X_1, \ldots, T_{\delta} X_{\delta})$,
and $\xi_j$ is the group ring operation defined above in $(\ref{xi})$. Let us for simplicity of 
notation make the identification $\Lambda \cong \mathcal{O}[[T_1, \ldots, T_{\delta}]]$ implicitly 
in the construction that follows.

\begin{lemma}\label{2.7} Given an integer $n \geq 0$,
let $\varepsilon$ denote the sign of $\left( -1 \right)^n$. Multiplication by
$\widetilde{\Omega}_{n}^{-\varepsilon}$ induces a bijection $\Lambda/\left( 
\Omega_{n}^{\varepsilon} \right) \longrightarrow \widetilde{\Omega}_{n}^{-\varepsilon}
\Lambda/\left( \Omega_n \right).$ \end{lemma}

\begin{proof} Cf. \cite[Lemma 2.7]{DI}, where the result is given
for $\delta=1$. A similar argument works here. That is,  let $g$ be any 
polynomial in $\Lambda$. We consider the map that sends 
$g \mapsto \widetilde{\Omega}_{n}^{-\varepsilon} g$. Observe that 
$\widetilde{\Omega}_{n}^{-\varepsilon} \Omega_{n}^{\varepsilon}
= \Omega_n$. Hence if $\Omega_{n}^{\varepsilon} \mid g$, then
$\widetilde{\Omega}_{n}^{\varepsilon} g \equiv 0 \mod \Omega_n$.
It follows that the map is injective. Since $\Lambda$ is a unique
factorization domain, the map is also seen to be surjective. \end{proof}

\begin{proposition}\label{2.8} Given a positive integer $n$, let
$\varepsilon$ denote the sign of $\left( -1 \right)^n$.

\begin{itemize}
\item[(i)] We have that $\Omega_{n}^{\varepsilon} \vartheta_{\Phi,
n} = 0.$

\item[(ii)] There exists a unique element $\Theta_{\Phi, n}^{\varepsilon} \in \Lambda /
\Omega_{n}^{\varepsilon} \Lambda$ such that $\vartheta_{\Phi, n} =
\widetilde{\Omega}_{n}^{-\varepsilon} \Theta_{\Phi,
n}^{\varepsilon}.$\end{itemize} \end{proposition}

\begin{proof} Cf. \cite[Proposition 2.8]{DI}. Let us first suppose
that $n >2$ is even. We then have that \begin{align*} \Omega_{n}^{+}
\vartheta_{\Phi, n} &= \Omega_{n-2}^{+} \xi_n \left( \vartheta_{\Phi, n} \right)\\
&= \Omega_{n-2}^{+} \xi_n \pi_{n-1} \left(
\vartheta_{\Phi, n} \right). \end{align*} Since
$a_{\mathfrak{p}}({\bf{f}}_0) = 0$, we obtain from Lemma
\ref{distribution} (i) that
\begin{align}\label{ordrel} \Omega_{n}^{+}
\vartheta_{\Phi, n} &= - \Omega_{n-2}^{+} \xi_n
\xi_{n-1}  \left( \vartheta_{\Phi, n-2} \right). \end{align} This allows us
to reduce to the case of $n=2$ by induction. Now, we find that
\begin{align*} \Omega_{2}^{+} \vartheta_{\Phi, 2} &= \left( T_1
\cdots T_{\delta} \right) \star \xi_2
\left( \vartheta_{\Phi, 2} \right) \\ &= (T_1 \cdots T_{\delta}) \star \xi_2
\pi_{2,1} \left( \vartheta_{\Phi, 2}  \right)
\\ &= - \left( T_1 \cdots T_{\delta} \right) \star \xi_1 \xi_2
\left( \vartheta_{\Phi, 0} \right). \end{align*} Observe that \begin{align*} (T_1
\cdots T_{\delta} ) \star \xi_1 \xi_2 &= (T_1 \cdots T_{\delta} )
\star \left( \Sigma_{p}(T_1 +1) \Sigma_{p^n}(T_1+1), \ldots,  
\Sigma_{p}(T_{\delta} +1) \Sigma_{p^n}(T_{\delta}+1)  \right)\\ &=
\Omega_2 (T_1, \ldots, T_{\delta}), \end{align*} which is $0$ in the
group ring $\mathcal{O}[G_{\mathfrak{p}^2}]$ by $(\ref{groupring})$. This
proves claim (i) for $n$ even. The case of $n$ odd can be shown in
the same way. To see (ii), observe that $\Omega_n =
\Omega_{n}^{\varepsilon} \widetilde{\Omega}_{n}^{-\varepsilon}.$
Using Lemma \ref{2.7}, deduce that any element in $\Lambda$
annihilated by $\Omega_{n}^{\varepsilon}$ must be divisible by
$\widetilde{\Omega}_{n}^{-\varepsilon}.$ We know that
$\Omega_{n}^{\varepsilon} \vartheta_{\Phi, n}=0$. Thus, we find that
$\vartheta_{\Phi, n}$ must be divisible by
$\widetilde{\Omega}_{n}^{-\varepsilon}.$ Since $\Lambda$ is a unique
factorization domain, this concludes the proof. \end{proof}
Using Proposition \ref{2.8}(ii), we may define elements
\begin{align*} \vartheta_{\Phi, n}^{+} &= \left( -1
\right)^{\frac{n}{2}} \cdot \Theta_{\Phi, n}^{+}  &\text{if}~ n \equiv 0 \mod 2\\
\vartheta_{\Phi, n}^{-} &= \left( -1 \right)^{\frac{n +1}{2}} \cdot
\Theta_{\Phi, n}^{-} &\text{if} ~n \equiv 1 \mod 2.
\end{align*}

\begin{lemma}\label{2.9}
The sequence $\lbrace \vartheta_{\Phi, n}^{\varepsilon} \rbrace_{n
\equiv (-1)^n \mod (2)}$ is compatible with respect to the natural
projection maps $\Lambda/\Omega_{n}^{\varepsilon} \longrightarrow
\Lambda/\Omega_{n-1}^{\varepsilon}.$
\end{lemma}

\begin{proof} Cf. \cite[Lemma 2.9]{DI}. Let us choose lifts to $\Lambda$
of the group ring elements $\vartheta_{\Phi, n}$ and $\Theta_{\Phi,
n}$ for all $n \geq 0$. We denote these lifts by the same symbols.
Let us first suppose that $n$ is even. Lemma \ref{distribution}(i)
implies that \begin{align*}\vartheta_{\Phi, n} = - \xi_{n-1}
\vartheta_{\Phi, n-2} \mod \Omega_{n-1}.\end{align*} Using Proposition
\ref{2.8}(ii), it follows that there exists a polynomial $f \in
\Lambda$ such that \begin{align}\label{id0}\widetilde{\Omega}_{n}^{-} 
\Theta_{\Phi,n}^{+} = -\xi_{n-1} \widetilde{\Omega}_{n-2}^{-}
\Theta_{n-2}^{+} + \Omega_{n-1}f .\end{align} Observe that
we have the identity \begin{align}\label{id1}  \Omega_{n-1}=
\Omega_{n-2}^{+} \widetilde{\Omega}_{n}^{-} .
\end{align} Observe that we also have the identity
\begin{align}\label{id2}
\widetilde{\Omega}_{n}^{-} = \xi_{n-1}
\widetilde{\Omega}_{n-2}.\end{align} Using $(\ref{id1})$, we may
cancel out by $\widetilde{\Omega}_{n}^{-}$ on either side of
$(\ref{id0})$ to obtain that \begin{align*}\Theta_{\Phi, n}^{+} = -
\Theta_{\Phi, n-2}^{+} + \Omega_{n-2}^{+} f \end{align*} by
$(\ref{id2})$. This proves the result for $n$ even. The case of $n$
odd can be shown in the same way. \end{proof} Using Lemma \ref{2.9},
we may define elements
\begin{align}\label{pmzeta} \vartheta_{\Phi}^{\pm} &= \ilim
n \vartheta_{\Phi, n}^{\pm} \in \ilim n \Lambda/\Omega_{n}^{\pm}.
\end{align} Observe again however that a different choice of sequence of
consecutive vertices $\lbrace \mathfrak{v}_{j} \rbrace_{j \geq 0}$
has the effect of multiplying $\vartheta_{\Phi}^{\pm}$ by some
element $\sigma \in G_{\mathfrak{p}^{\infty}}$. As in the ordinary
case, we correct this by making the following

\begin{definition} Let $\mathcal{L}_{\mathfrak{p}}(\Phi, K)^{\pm} = 
\vartheta_{\Phi}^{\pm} \cdot \left( \vartheta_{\Phi}^{\pm} \right)^*.$
\end{definition} Note that $\mathcal{L}_{\mathfrak{p}}(\Phi, K)^{\pm}$ is then
a well-defined element of $\mathcal{O}[[G_{\mathfrak{p}^{\infty}}]].$
\end{remark}

\begin{remark}[$p$-adic $L$-functions.] In both cases on ${\bf{f}}_0$, we refer to 
the associated element $\mathcal{L}_{\mathfrak{p}}(\Phi, K) $ or  
$\mathcal{L}_{\mathfrak{p}}(\Phi, K)^{\pm} \in  \Lambda$, with $\Phi$ the 
eigenform associated to ${\bf{f}}_0$ by Proposition \ref{p-JLC}, the
{\it{(quaternionic) $\mathfrak{p}$-adic $L$-function}} associated to ${\bf{f}}_0$ and 
$G_{\mathfrak{p}^{\infty}} = \Gal(K_{\mathfrak{p}^{\infty}}/K)$. \end{remark}

\begin{remark}[Interpolation properties.]

Recall that we let $\Lambda$ denote the $\mathcal{O}$-Iwasawa algebra 
$\mathcal{O}[[G_{\mathfrak{p}^{\infty}}]]$. Let $\rho$ be any 
finite order character of the Galois group $G_{\mathfrak{p}^{\infty}}$.
Let $\rho \left( \mathcal{L}_{\mathfrak{p}}(\Phi, K) \right)$ denote
the specialization of $\mathcal{L}_{\mathfrak{p}}(\Phi, K)$ to
$\rho$. To be more precise, a continuous homomorphism $\rho:
G_{\mathfrak{p}^{\infty}} \longrightarrow {\bf{C}}_{p}$ extends to
an algebra homomorphism $\Lambda \longrightarrow {\bf{C}}_p$ by the
rule \begin{align} \label{Ghomo}\rho(\lambda) = 
\int_{G_{\mathfrak{p}^{\infty}}} \rho(x) d\lambda(x),
\end{align} with $d\lambda$ the $\mathcal{O}$-valued measure of 
$G_{\mathfrak{p}^{\infty}}$ associated to an element $\lambda \in \Lambda.$ 

\begin{remark}
Note that a product of elements
$\lambda_1 \lambda_2 \in \Lambda$ corresponds to convolution of measures
$d \left( \lambda_1 \boxtimes \lambda_2 \right)$ under specialization, i.e.
\begin{align*}\rho(\lambda_1 \lambda_2) =
\int_{G_{\mathfrak{p}^{\infty}}} \rho(x) d \left( \lambda_1
\boxtimes \lambda_2 \right) = \int_{G_{\mathfrak{p}^{\infty}}}
\left( \int_{G_{\mathfrak{p}^{\infty}}} \rho(x+y) d\lambda_1(x)
\right) d\lambda_2(y).\end{align*} \end{remark}
We now state the following consequence of Theorem \ref{W}. 
Let $\mathcal{L}_{\mathfrak{p}}(\Phi,K)^{\star}$ denote any of the 
$\mathfrak{p}$-adic $L$-functions $\mathcal{L}_{\mathfrak{p}}(\Phi, K),$
$\mathcal{L}_{\mathfrak{p}}(\Phi, K)^{+},$ or
$\mathcal{L}_{\mathfrak{p}}(\Phi, K)^{-}.$

\begin{theorem}\label{basicint} Fix embeddings
$\overline{\bf{Q}} \rightarrow \overline{\bf{Q}}_p$ and
$\overline{\bf{Q}}_p \rightarrow {\bf{C}}$. Let $\rho$ be any finite
order character of $G_{\mathfrak{p}^{\infty}}$ that factors through
$G_{\mathfrak{p}^m}$ for some integer $m \geq 1$. Let us view the
values of $\rho$ and $d\mathcal{L}_{\mathfrak{p}}(\Phi,
K)^{\star}$ as complex values via $\overline{\bf{Q}}_p
\rightarrow {\bf{C}}$, in which case we let $\vert \rho \left(
\mathcal{L}_{\mathfrak{p}}(\Phi, K) \right)^{\star} \vert$
denote the complex absolute value of the specialization $\rho \left(
\mathcal{L}_{\mathfrak{p}}(\Phi, K)^{\star} \right) $. We have
the following interpolation formulae in the notations of Theorem
\ref{W} above.

\begin{itemize}
\item[(i)] If $\Phi$ is $\mathfrak{p}$-ordinary, then
\begin{align*} \vert \rho\left(\mathcal{L}_{\mathfrak{p}}(\Phi, K)
\right) \vert &= \frac{ \alpha_{\mathfrak{p}}^{-2m} \cdot
\zeta_F(2)}{ 2 \cdot L(\pi, \operatorname{ad}, 1)} \\
&\times  \left[ L(\pi, \rho, 1/2) \cdot L(\pi, \rho^{-1}, 1/2)
\cdot \prod_{v \nmid \infty} \alpha(\Phi_v, \rho_v) \cdot
\alpha(\Phi_v, \rho_v^{-1}) \right]^{\frac{1}{2}}. \end{align*}

\item[(ii)] If $\Phi$ is $\mathfrak{p}$-supersingular, then
\begin{align*} \vert \rho\left(
\mathcal{L}_{\mathfrak{p}}(\Phi, K)^{\pm} \right) \vert &=
\frac{\zeta_F(2)}{ 2 \cdot L(\pi, \operatorname{ad}, 1)} \\
&\times  \left[ L(\pi, \rho, 1/2) \cdot L(\pi, \rho^{-1}, 1/2)
\cdot \prod_{v \nmid \infty} \alpha(\Phi_v, \rho_v) \cdot
\alpha(\Phi_v, \rho_v^{-1}) \right]^{\frac{1}{2}}. \end{align*}
\end{itemize} Note that the values on the right hand sides of
$(i)$ and $(ii)$ are both algebraic as a consequence of Theorem
\ref{W}, and hence can be viewed as values in $\overline{\bf{Q}}_p$
via our fixed embedding $\overline{\bf{Q}} \rightarrow
\overline{\bf{Q}}_p$.

\end{theorem} Observe in particular that the specialization
$\rho\left( \mathcal{L}_{\mathfrak{p}}(\Phi, K)^{\star}\right)$
vanishes if and only if the complex central value $L(\pi, \rho,
1/2)$ vanishes. Hence, we obtain from Theorem \ref{CV} 
(or the stronger result deduced in Corollary \ref{strongCV} above) 
the following important

\begin{corollary}\label{nontriviality} The element 
$\mathcal{L}_{\mathfrak{p}}(\Phi, K)^{\star}$ does
not vanish identically in $\mathcal{O}[[G_{\mathfrak{p}^{\infty}}]]$.
\end{corollary} To prove Theorem \ref{basicint}, let us
first consider the following basic result. Recall that given an
element $\lambda \in \Lambda$, we let $\lambda^*$ denote the image
of $\lambda$ under the involution sending $\sigma \mapsto
\sigma^{-1} \in G_{\mathfrak{p}^{\infty}}$.

\begin{lemma}\label{inv} We have that $\rho \left( \lambda^{*} \right) =
\rho^{-1} \left( \lambda \right)$ for any
 $\lambda \in \Lambda$.
\end{lemma}

\begin{proof}
Since $\rho: G_{\mathfrak{p}^{\infty}} \longrightarrow
{\bf{C}}^{\times}$ is a homomorphism of groups, we have that
$\rho(\sigma^{-1}) = \rho(\sigma)^{-1}$ for any $\sigma \in
G_{\mathfrak{p}^{\infty}}$, as a consequence of the basic identities
\begin{align*} \rho(\sigma) \rho(\sigma^{-1}) = \rho(\sigma)
 \rho(\sigma)^{-1} = 1.\end{align*} Using the definition of
$\lambda^{*},$ we then find that
\begin{align*}\rho \left( \lambda^{*} \right)
&= \int_{G_{\mathfrak{p}^{\infty}}} \rho(\sigma) d\lambda^{*} (\sigma)
&= \int_{G_{\mathfrak{p}^{\infty}}} \rho(\sigma^{-1}) d\lambda(\sigma) 
&= \int_{G_{\mathfrak{p}^{\infty}}}
\rho^{-1} (\sigma) d\lambda(\sigma) &= \rho^{-1}(\lambda).
\end{align*} \end{proof} We now prove Theorem \ref{basicint}.

\begin{proof} Suppose first that the eigenform $\Phi$ is
$\mathfrak{p}$-ordinary, hence that the
$T_{\mathfrak{p}}$-eigenvalue of $\Phi$ is a $p$-adic unit with respect to 
our fixed embedding $\overline{\bf{Q}} \rightarrow \overline{{\bf{Q}}}_p$. 
Recall that in this case, we define $\mathcal{L}_{\mathfrak{p}}(\Phi, K) = 
\theta_{\Phi} \theta_{\Phi}^{*}$ as in $(\ref{ordzeta})$. Let $\rho$ be a 
finite order character of $G_{\mathfrak{p}^{\infty}}$ that factors through
$G_{\mathfrak{p}^m}$. We have by definition that \begin{align*}
\rho\left( \theta_{\Phi}\right) &= \int_{G_{\mathfrak{p}^{\infty}}}
\rho(\sigma) \cdot d\theta_{\Phi}(\sigma) \\ &=
\alpha_{\mathfrak{p}}^{-m} \cdot \int_{G_{\mathfrak{p}^{\infty}}}
\rho(\sigma) \cdot \Phi_{K, m}(\sigma) \\ &=
\alpha_{\mathfrak{p}}^{-m} \cdot \int_{G_{\mathfrak{p}^{\infty}}}
\rho(\sigma) \cdot \Phi \left( \iota_{\mathfrak{p}} \circ
\eta_{\mathfrak{p}}^{-1} \circ \widehat{\Psi} \circ \rec_K^{-1}
(\sigma) \star \mathfrak{e}_m \right).
\end{align*} Here, $\mathfrak{e}_m$ denotes the $m$-th directed
edge in the fixed sequence $\lbrace \mathfrak{e}_j\rbrace_{j\geq1}$
defined above. Let $\mathfrak{e}_{m}^{\sigma}$ to denote the
directed edge defined by $\eta_{\mathfrak{p}}^{-1} \circ \widehat{\Psi} 
\circ \rec_K^{-1}(\sigma) \star \mathfrak{e}_m$, where $\star$ denotes the 
induced conjugation action. We argue that the value 
$\Phi(\mathfrak{e}_m^{\sigma})$ can be identified with the value
$\Phi(t)$, where $t = \rec_{K}^{-1}(\sigma)$, and $\Phi(t)$ 
denotes the evaluation at $t$ of the corresponding eigenform 
$\Phi \in \mathcal{S}_2^B(H;\mathcal{O}).$ That is, recall from
the discussion above that the action of the Galois
group $G_{\mathfrak{p}^m}$ on the Bruhat-Tits tree 
$\mathcal{T}_{\mathfrak{p}}$ factors through the induced conjugation
action by the quotient $K_{\mathfrak{p}}^{\times}/F_{\mathfrak{p}}^{\times}$.
In particular, the quotient $(K_{\mathfrak{p}}^{\times}/F_{\mathfrak{p}}^{\times})/\UU_m$
acts simply transitively on the set of vertices of distance $m$ away from the 
geodesic $J$ of vertices fixed by the maximal compact subgroup
$\UU_0 \subseteq K_{\mathfrak{p}}^{\times}/F_{\mathfrak{p}}^{\times}$. Now, $\mathfrak{e}_m$
is given by the intersection of $2$ maximal orders corresponding to vertices
$(\mathfrak{v}_{m-1}, \mathfrak{v}_m)$ say, where $d(\mathfrak{v}_j, J) = j$.  
Using that $G_{\mathfrak{p}^m}$ acts simply transitively, we deduce that 
$\mathfrak{e}_m^{\sigma}$ is given by the intersection of $2$ maximal orders
corresponding to vertices $(\mathfrak{v}_{m-1}, \mathfrak{v}_m')$, where
$\mathfrak{v}_m' =  \eta_{\mathfrak{p}}^{-1}\circ 
\widehat{\Psi} \circ \rec_K^{-1}(\sigma) \star \mathfrak{v}_m$ is the 
vertex obtained from the action of Galois, having $d(\mathfrak{v}_m',J)=m$.
Now, recall that by Lemma \ref{lattices-orders}, we have a bijection between
the set of maximal orders of $\M(F_{\mathfrak{p}})$ and the set of homothety 
classes of full rank lattices of $F_{\mathfrak{p}} \oplus F_{\mathfrak{p}}$. The origin
vertex $\mathfrak{v}_0$ corresponds to the class of the lattice 
$\mathcal{O}_{F_{\mathfrak{p}}}\oplus \mathcal{O}_{F_{\mathfrak{p}}}$ under any such
bijection. Since $d(\mathfrak{v}_0, \mathfrak{v}_{m-1}) = m-1$ by construction,
we may take $$L_1 = \mathcal{O}_{F_{\mathfrak{p}}} \oplus 
\pi_{\mathfrak{p}}^{m-1}\mathcal{O}_{F_{\mathfrak{p}}}$$ as a lattice
representative for the class corresponding to the vertex $\mathfrak{v}_m$.
Here, $\pi_{\mathfrak{p}}$ is a fixed uniformizer of $F_{\mathfrak{p}}$.
Similarly, we may take $$L_2 = \mathcal{O}_{F_{\mathfrak{p}}} \oplus
\pi_{\mathfrak{p}}^m \mathcal{O}_{F_{\mathfrak{p}}}$$ as a lattice
representative for the class corresponding to the vertex $\mathfrak{v}_m'$. 
We then consider the value of the lattice class function ${\bf{c}}_{\Phi}([L_1], [L_2])$.
Recall that by definition, ${\bf{c}}_{\Phi}([L_1], [L_2]) = \Phi(g_L)$. Here,
we have fixed a pair of representatives $(L_1, L_2)$ for the pair of classes
$([L_1], [L_2])$, and $g_L$ is any matrix in $\GL(F_{\mathfrak{p}})$ such that \begin{align*} 
g_L \left( \mathcal{O}_{F_{\mathfrak{p}}} \oplus 
\mathcal{O}_{F_{\mathfrak{p}}}\right) &= L_1 \\ 
g_L \left( \mathcal{O}_{F_{\mathfrak{p}}} \oplus \pi_{\mathfrak{p}}^{m-1}
\mathcal{O}_{F_{\mathfrak{p}}}\right) &= L_2.\end{align*} It is then clear
that we can take \begin{align*} g_L &= \left( \begin{array}{cc} 1 & 0\\
0 & \pi_{\mathfrak{p}}^{m-1}\end{array}\right)\end{align*} for this matrix
representative. We claim it is also clear that this matrix $g_L$ is 
contained in the image of the local optimal embedding \begin{align*}
\begin{CD} \left(K_{\mathfrak{p}}^{\times}/F_{\mathfrak{p}}^{\times}\right)/\UU_m
@>{\Psi_{\mathfrak{p}}}>> B_{\mathfrak{p}}^{\times} \cong \GL(F_{\mathfrak{p}}).
\end{CD}\end{align*} Granted this claim, we see that the matrix $g_L$ 
factors through the action of the Galois group $G_{\mathfrak{p}^m}$ on 
the directed edge set. In particular, we deduce by transitivity of the 
action that \begin{align*} \lbrace \Phi(\mathfrak{e}_m^{\sigma}) \rbrace_{
\sigma \in G_{\mathfrak{p}^m}} = \lbrace \Phi(t)\rbrace_{t \in 
\rec_K^{-1}(G_{\mathfrak{p}^m})}, \end{align*} where $\Phi(t)$ denotes the
evaluation of the global eigenform $\Phi \in \mathcal{S}_2^B(H;\mathcal{O})$
on a (torus) class $t$. In particular, we deduce that  
$\Phi(\mathfrak{e}_m^{\sigma}) = \Phi(\rec_K^{-1}(\sigma)) = \Phi(t)$.
Granted this identification, the specialization $\rho
\left( \theta_{\Phi} \right)$ is then given by
\begin{align*} \alpha_{\mathfrak{p}}^{-m} \cdot
\int_{\rec_K^{-1}\left( G_{\mathfrak{p}^{\infty}} \right)} \rho(t)
\cdot \Phi \left( t \right) dt,
\end{align*} with $ t = \rec_K^{-1}\left( \sigma \right).$ Here,
$dt$ denotes the counting Haar measure, which coincides with the
Tamagawa measure. We are now in a position to invoke the special
value formula of Theorem \ref{W} above directly. That is, since $\rho$ extends 
to an algebra homomorphism $\Lambda \longrightarrow {\bf{C}}_p $, 
it follows from Lemma \ref{inv} that $\rho \left( \theta_{\Phi} 
\theta_{\Phi}^*\right) = \rho(\theta_{\Phi}) \cdot \rho^{-1}(\theta_{\Phi})$.
Hence, we find that \begin{align*} \vert \rho \left(
\mathcal{L}_{\mathfrak{p}}(\Phi, K) \right) \vert &= \vert
\rho(\theta_{\Phi}) \vert \cdot \vert \rho^{-1}(\theta_{\Phi}) \vert
= \lvert \l(\Phi, \rho)
\vert \cdot \vert l(\Phi, \rho^{-1}) \vert. \end{align*} Here,
$l(\Phi, \rho)$ denotes the period integral defined in
$(\ref{periodintegral})$. The result then
follows directly from Theorem \ref{W}. In the case that $\Phi$ is
$\mathfrak{p}$-supersingular, hence that the
$T_{\mathfrak{p}}$-eigenvalue of $\Phi$ is zero, the same argument
gives the analogous interpolation formula for $\vert \rho \left(
\mathcal{L}_{\mathfrak{p}}(\Phi, K)^{\pm} \right) \vert$. \end{proof}\end{remark}

\begin{remark}[The invariant $\mu$.]

We now give an expression for the Iwasawa $\mu$-invariant
associated to any of the $\mathfrak{p}$-adic $L$-functions
$\mathcal{L}_{\mathfrak{p}}(\Phi, K)^{\star}$, following the
method of Vatsal \cite{Va2}. Recall that we let $\Lambda$ denote
the $\mathcal{O}$-Iwasawa algebra of $G_{\mathfrak{p}^{\infty}}$, which
is the completed group ring $\mathcal{O}[[G_{\mathfrak{p}^{\infty}}]]$. 
Recall as well that we define the {\it{$\mu$-invariant}} $\mu(Q)$ of
an element $Q \in \Lambda$ to be the largest exponent $c$ such that $Q
\in \mathfrak{P}^c\Lambda$. \begin{definition} Given an 
eigenform $\Phi \in \mathcal{S}_2\left( \coprod_{i=1}^{h} \Gamma_i
\backslash \mathcal{T}_{\mathfrak{p}}; \mathcal{O} \right)$, let $\nu =
\nu_{\Phi}$ denote the largest integer such that $\Phi$ is congruent
to a constant modulo $\mathfrak{P}^{\nu}$.\end{definition}

\begin{theorem}\label{mu}
The $\mu$-invariant $\mu(\mathcal{L}_{\mathfrak{p}}(\Phi,
K)^{\star})$ is given by $2\nu$.
\end{theorem}

\begin{proof} See Vatsal \cite[Proposition 4.1, $\S$ 4.6]{Va2},
which proves the analogous result for $F={\bf{Q}}$. Let us assume 
first that $\Phi$ is $\mathfrak{p}$-ordinary, hence that the image of 
its $T_{\mathfrak{p}}$-eigenvalue under our fixed embedding 
$\overline{\bf{Q}} \rightarrow \overline{{\bf{Q}}}_p$ is a $p$-adic 
unit. Recall that in this case, we define $\mathcal{L}_{\mathfrak{p}}(\Phi,
K) = \theta_{\Phi} \theta_{\Phi}^{*}$ by the formula 
$(\ref{ordzeta})$. Let $\rho$ be any ring class character of $K$
that factors through $G_{\mathfrak{p}^{m}}$ for some integer $m \geq
1$. Observe that by definition, we have the congruence $\rho 
\left( \theta_{\Phi, m} \right) \equiv 0 \mod \mathfrak{P}^{\nu}$. 
Hence, we find that $\mu \left( \theta_{\Phi} \right) \geq \nu$. 
Our approach is now to find a coefficient in the power series 
expansion for $\theta_{\Phi}$ having $\mathfrak{P}$-adic valuation 
at most $\nu$. Let us then write the completed group ring element
$\theta_{\Phi}$ as \begin{align*} \theta_{\Phi} = \ilim j \left( 
\sum_{\sigma \in G_{\mathfrak{p}^{j}}} c_j(\sigma) \cdot \sigma \right) 
\in \Lambda .\end{align*} Writing ${\bf{1}}$ for the identity
in $G_{\mathfrak{p}^{\infty}}$, we then obtain from
$(\ref{ordzeta})$ the expression $$c_{\infty}({\bf{1}}) := \ilim j
c_j({\bf{1}}) = \ilim j \left( \alpha_{\mathfrak{p}}^{-j} \cdot
[{\bf{1}}, \mathfrak{e}_j]_{\Phi} \right)$$ for the constant term in
the power series expansion of $\theta_{\Phi}$ We claim that
$c_{\infty}({\bf{1}})$ has $\mathfrak{P}$-adic valuation at most
$\nu$. Equivalently, we claim that there exists a sequence of
directed edges $\lbrace \mathfrak{e}_j' \rbrace_{j \geq 1}$ such
that $$\ilim j \left( [{\bf{1}}, \mathfrak{e}_j]_{\Phi} \right) \neq
\ilim j \left( [{\bf{1}}, \mathfrak{e}_j']_{\Phi} \right) \mod
\mathfrak{P}^{\nu + 1}.$$ Indeed, suppose otherwise. Then, for any
sequence of directed edges $\lbrace \mathfrak{e}_j' \rbrace_{j \geq
1},$ we would have that \begin{align*} \Phi(\mathfrak{e}_j) \equiv
\Phi(\mathfrak{e}_j') \mod \mathfrak{P}^{\nu+1}.\end{align*} In
particular, it would follow that $\Phi$ were congruent to a constant
modulo $\mathfrak{P}^{\nu +1}$, giving the desired contradiction.
Using the same argument for the element $\theta_{\Phi}^{*}$, we find
that $\mu\left(\mathcal{L}_{\mathfrak{p}}(\Phi, K)\right) = 2\nu$.
Assume now that $\Phi$ is $\mathfrak{p}$-supersingular, hence that
its $T_{\mathfrak{p}}$-eigenvalue is $0$. We claim that for each of
the $\mathfrak{p}$-adic $L$-functions
$\mathcal{L}_{\mathfrak{p}}(\Phi, K)^{\pm}$, we have that $\mu\left(
\mathcal{L}_{\mathfrak{p}}(\Phi, K)^{\pm} \right) =
\mu\left(\vartheta_{\Phi} \vartheta_{\Phi}^{*} \right)$, as the
contribution of trivial zeroes from $\Omega_n$ will not affect the
$\mathfrak{P}$-adic valuation. The same argument given above then
shows that $\mu\left(\vartheta_{\Phi} \vartheta_{\Phi}^{*} \right) =
2\nu$, which concludes the proof. \end{proof}\end{remark}

\section{Howard's criterion}

We conclude with the nonvanishing criterion Howard, \cite[Theorem 3.2.3(c)]{Ho}. 
This criterion, if satisfied, has important consequences for the associated 
Iwasawa main conjecture by the combined works of Howard \cite[Theorem 3.2.3]{Ho},
and Pollack-Weston \cite{PW}, as explained in Theorem \ref{IMC} above for the case
of $F={\bf{Q}}$. If also has applications to the analogous Iwasawa main conjectures
for general totally real fields, as explained in Theorem 1.3 of the sequel work \cite{VO2}.

Fix a Hilbert modular eignform ${\bf{f}} \in \mathcal{S}_2(\mathfrak{N})$,
with $\mathfrak{N} \subset \mathcal{O}_F$ an integral ideal having the 
factorization $(\ref{factorization})$. Let us for simplicity assume that $\mathfrak{N}$
is prime to the relative discriminant of $K$ over $F$. Given a positive integer $k$, 
let us define a set of admissible primes $\mathfrak{L}_k$ of $\mathcal{O}_F$, 
all of which are inert in $K$, with the condition that for any ideal
$\mathfrak{n}$ in the set $\mathfrak{S}_k$ of squarefree products
of primes in $\mathfrak{L}_k$, there exists a nontrivial eigenform
${\bf{f}}^{(\mathfrak{n})}$ of level $\mathfrak{nN}$ such
that \begin{align}\label{congruence} {\bf{f}}^{(\mathfrak{n})}
\equiv {\bf{f}} \mod \mathfrak{P}^k.\end{align} Here, 
$(\ref{congruence})$ denotes a congruence of Hecke eigenvalues. 
Let $\mathfrak{S}_{k}^{+} \subset \mathfrak{S}_{k}$ 
denote the subset of ideals $\mathfrak{n}$ for which 
$\omega_{K/F}(\mathfrak{n}\mathfrak{N}) = -1$, where recall $\omega_{K/F}$ denotes 
the quadratic Hecke character associated to $K/F$. Equivalently, we can let 
$\mathfrak{S}_{k}^{+} \subset \mathfrak{S}_{k}$ 
denote the subset of ideals $\mathfrak{n}$ for which the root number of the complex
$L$-function $L({\bf{f}}, K, s)$ is $+1$. Note that this set 
$\mathfrak{S}_k^{+}$ includes the so called ``empty
product" corresponding to $1$. Given an ideal $\mathfrak{n} \in
\mathfrak{S}_{k}^{+}$, we have an associated $p$-adic $L$-function
$\mathcal{L}_{\mathfrak{p}}({\bf{f}}^{(\mathfrak{n})}, K)^{\star} = 
\left(\theta_{{\bf{f}}^{(\mathfrak{n})}} \theta_{{\bf{f}}^{(\mathfrak{n})}}^* \right)^{\star}$.
Let us then for simplicity write $\lambda_{\mathfrak{n}}$ 
to denote the associated completed group ring element 
$\theta_{{\bf{f}}^{(\mathfrak{n})}} $, with $\lambda_1$ the 
base element $\theta_{{\bf{f}}}$. Let $\mathfrak{Q}$ be any height one prime ideal of 
$\Lambda = \mathcal{O}[[G_{\mathfrak{p}^{\infty}}]].$ We say that {\it{Howard's criterion for 
${\bf{f}}$ and $K$}} holds at $\mathfrak{Q}$ if there exists an integer 
$k_0$ such that for each integer $j \geq k_0$, the set \begin{align}\label{weird} 
\lbrace \lambda_{\mathfrak{n}} \in \Lambda/(\mathfrak{P}^j): \mathfrak{n} \in \mathfrak{S}_{j}^{+} \rbrace\end{align} 
contains at least one element $\lambda_{\mathfrak{n}}$ with nontrivial image in 
$\Lambda/(\mathfrak{Q}, \mathfrak{P}^{k_0})$. Following the result of Howard \cite[Theorem 3.2.3 (c)]{Ho}, as well
as the generalization given in \cite[Theorem 1.3]{VO2}, we make the following

\begin{conjecture}\label{HC}
Howard's criterion for ${\bf{f}}$ and $K$ holds at any height one prime ideal $\mathfrak{Q}$ of $\Lambda$.
\end{conjecture}

\begin{remark} Note that Conjecture \ref{HC} holds trivially if $\ord_{\mathfrak{Q}}(\lambda_{\mathfrak{n}}) =0$
for some $\lambda_{\mathfrak{n}}$ in the set $(\ref{weird})$. It is also easy to see that Conjecture \ref{HC} 
holds at the height one prime defined by $\mathfrak{Q} = (\mathfrak{P})$, using the characterization of the 
$\mu$-invariant given in Theorem \ref{mu} above. Observe moreover that Conjecture \ref{HC} holds trivially
for {\it{all}} height one primes of $\Lambda$ if one of the elements $\lambda_{\mathfrak{n}}$
in the set $(\ref{weird})$ is a unit in $\Lambda$. Hence, as explained in Theorem \ref{IMC} above,
or more generally in Theorem 1.3 of the sequel paper \cite{VO2}, this condition would often be strong enough to imply
the {\it{full}} associated Iwasawa main conjecture, i.e. that the equality of ideals $(\ref{fullequality})$ indeed holds in 
Conjecture \ref{mainconjecture} above. \end{remark}

We conclude this discussion with a reformulation Conjecture \ref{HC} at the height one prime of $\Lambda$ defined 
by $\mathfrak{Q} = (\gamma_1 -1,\ldots, \gamma_{\delta} -1)$ into a conjecture about the nonvanishing of central
values of complex Rankin-Selberg $L$-functions. Let ${\bf{1}}$ denote the trivial character of the Galois group 
$G_{\mathfrak{p}^{\infty}}$. Recall that we write $L({\bf{f}}, K, s)$ to denote the Rankin-Selberg $L$-function of 
${\bf{f}}$ times the theta series associated to $K$, normalized to have central value at $s=1$. Consider the 
following easy result.

\begin{lemma}\label{HCr2}
Howard's criterion holds at $\mathfrak{Q} = (\gamma_1 -1, \ldots, \gamma_{\delta}-1)$ if and only if there 
exists an integer $k_0$ such that for all integers $j \geq k_0$, the set $\mathfrak{S}_j^{+}$ contains an ideal 
$\mathfrak{n}$ such that the associated central value $L( {\bf{f}}^{(\mathfrak{n})}, K,1)$ does not vanish.
\end{lemma}

\begin{proof} We claim that $\mathfrak{Q}$ divides an element $\lambda \in \Lambda$ if and only if the 
specialization \begin{align*}{\bf{1}}(\lambda) = \int_{G_{\mathfrak{p}^{\infty}}}
{\bf{1}}(\sigma) d\lambda(\sigma)\end{align*} does not vanish. This can be seen by translating to 
the power series description of $\Lambda$. The claim then follows immediately from
our interpolation formula for these $\mathfrak{p}$-adic $L$-functions (Theorem \ref{basicint}), using the
central value formula of Theorem \ref{W}. \end{proof} We therefore conclude this note with the following
intriguing

\begin{conjecture}\label{HCnv}

Let ${\bf{f}} \in \mathcal{S}_2(\mathfrak{N})$ be a cuspidal Hilbert modular
eigenform, and $K/F$ a totally imaginary quadratic extension for which the
root number of $L({\bf{f}}, K, s)$ equals $+1$. Then,
there exists a positive integer $k_0$ such that the 
following property is satisfied:
for each integer $j \geq k_0$, there exists an ideal 
$\mathfrak{n} \in \mathfrak{S}_{j}^{+}$ such that the central
value $L({\bf{f}}^{(\mathfrak{n})}, K, 1)$ does not vanish.
\end{conjecture}

\begin{remark}[Acknowledgement.] It is a pleasure to thank the many people with whom I 
discussed various aspects of this work, in particular Kevin Buzzard, John Coates, Henri Darmon, 
Ben Howard, Matteo Longo, Chung Pang Mok, Philippe Michel, Rob Pollack, Tony Scholl, Xin Wang
and Shou-Wu Zhang. It is also a pleasure to thank the anonymous referee for helpful comments
and suggestions. \end{remark}


\begin{thebibliography}{10}

\bibitem{BD2}
M.~Bertolini and H.~Darmon,
\newblock {\it{Heegner points on Mumford-Tate curves}}
\newblock Inventiones mathematicae, {\bf{126}} (1996), 413-456.

\bibitem{BD}
M.~Bertolini and H.~Darmon, 
\newblock {\it{Iwasawa's Main Conjecture for elliptic curves over anticyclotomic ${\bf{Z}}_p$-extensions}},
\newblock Annals of Math., {\bf{162}} (2005), 1 - 64.

\bibitem{BB}
N.~Bourbaki,
\newblock {\it{\'El\'ements de math\'ematique, Fasc. XXX1, Algebre commutative, Chapitre 7: Diviseurs,}}
\newblock Act. Scientifiques et Industrielles, {\bf{1314}} Paris, Hermann (1965).

\bibitem{CV}
C.~Cornut and V.~Vatsal,
\newblock {\it{Nontriviality of Rankin-Selberg $L$-functions and CM points}}, 
\newblock $L$-functions and Galois Representations, Ed. Burns, Buzzard and 
Nekov{\'a}\v{r}, Cambridge University Press (2007), 121- 186

\bibitem{DI}
H.~Darmon and A.~Iovita,
\newblock{\it{The anticyclotomic Main Conjecture for elliptic curves at supersingular primes,}}
\newblock Jour. Inst. Math. Jussieu, 7 {\bf{2}} (2008), 291 - 325.

\bibitem{DT}
F.~Diamond and R.~Taylor,
\newblock {\it{Nonoptimal levels of mod $l$ modular representations}}, 
\newblock Inventiones mathematicae {\bf{115}} (1994), 435-462.

\bibitem{Ga}
P.B.~Garrett,
\newblock {\it{Holomorphic Hilbert Modular Forms}},
\newblock Brooks/Cole Publishing Company, (1990).

\bibitem{Ge}
G.~van der Geer,
\newblock {\it{Hilbert Modular Surfaces}},
\newblock Springer-Verlag (1980).

\bibitem{Go}
E.~Goren,
\newblock {\it{Lectures on Hilbert Modular Varieties and Modular Forms,}}
\newblock American Mathematical Society, (2001).

\bibitem{Ho0}
B.~Howard,
\newblock{\it{Iwasawa theory of Heegner points on abelian varieties of $\GL$-type,}}
\newblock Duke Mathematical Journal {\bf{124}} No. 1 (2004), 1-45.

\bibitem{Ho}
B.~Howard,
\newblock{\it{Bipartite Euler systems}},
\newblock J. reine angew. Math. {\bf{597}} (2006).

\bibitem{JL}
H.~Jacquet and R.P.~Langlands,
\newblock {\it{Automorphic forms on $GL(2),$}}
\newblock Lecture Notes in Mathematics {\bf{278}},
\newblock Springer-Verlag, Berlin, 1970.

\bibitem{Lo2}
M.~Longo,
\newblock{\it{Anticyclotomic Iwasawa's Main Conjecture for Hilbert Modular Forms}},
\newblock Commentarii Mathematici Helvetici (to appear).

\bibitem{Mok}
C.-P.~Mok,
\newblock {\it{Heegner points and $p$-adic $L$-functions for elliptic curves over certain totally real fields}}, 
Commentarii Mathematici Helvetici (to appear).

\bibitem{Mok2}
C.P.~Mok,
\newblock {\it{The exceptional zero conjecture for Hilbert modular forms}}, 
\newblock Compositio Mathematica {\bf{145}} (2009), 1 - 55.

\bibitem{Ne}
J.~Nekovar,
\newblock {\it{Level raising and anticyclotomic Selmer groups for Hilbert modular forms of weight $2$}},
\newblock Canadian Journal of Mathematics (to appear).

\bibitem{Poll}
R.~Pollack,
\newblock{\it{On the $p$-adic $L$-functions of a modular form at a supersingular prime}},
\newblock Duke Mathematical Journal {\bf{118}} (200) 523 - 558.

\bibitem{PW}
R.~Pollack and T.~Weston,
\newblock{\it{On anticyclotomic $\mu$-invariants of modular forms}},
\newblock Compositio Mathematica (to appear).

\bibitem{Raj}
A.~Rajaei,
\newblock{\it{On the levels of mod $l$ Hilbert modular forms,}}
\newblock J. reine angew. Math., {\bf{537}} (2001), 35-65.

\bibitem{Sai}
H.~Saito,
\newblock{\it{ On Tunnell's formula for characters of $GL(2)$}},
\newblock Compositio Mathematica {\bf{85}} (1985), 99-108.

\bibitem{Sh2} G.~Shimura, {\it{The special values of zeta functions associated with Hilbert modular forms}},
\newblock Duke Mathematical Journal {\bf{45}} (1978), 637-679.

\bibitem{Tu}
J.~Tunnell,
\newblock {\it{Local $\epsilon$-factors and characters of $GL(2)$}},
\newblock Amererican Journal of Mathematics {\bf{105}} No. 6 (1983),
1277-1307.

\bibitem{VO2}
J.~Van Order, {\it{On dihedral main conjectures of Iwasawa theory for Hilbert modular eigenforms}}, 
\newblock Canadian Journal of Mathematics (to appear), available at \texttt{http://arxiv.org/abs/1112.3823}

\bibitem{Va2}
V.~Vatsal,
\newblock {\it{Special Values of Anticyclotomic
$L$-functions}}, \newblock Duke Mathematical Journal {\bf{116}} (2003), 219-261.

\bibitem{Va}
V.~Vatsal,
\newblock {\it{Uniform distribution of Heegner points}},
\newblock Inventiones mathematicae {\bf{148}} 1-46 (2002).

\bibitem{Vi}
M.-F.~ Vigneras,
\newblock {\it{Arithm{\'e}tique des Alg{\`e}bres des Quaternions}},
\newblock Lecture Notes in Mathematics {\bf{800}}, Springer-Verlag, New
York, 1980.

\bibitem{Wa}
J.-P.~Waldspurger,
\newblock {\it{Sur les valeurs de certaines fonctions $L$
automorphes en leur centre de sym\'etrie}},
\newblock Compositio Mathematica {\bf{54}} (1985), 173-242.

\bibitem{YZ^2}
X.~Yuan, S.-W.~Zhang, and W.~Zhang,
\newblock {\it{Gross-Zagier formula on Shimura curves}}
\newblock (preprint), available at \texttt{http://www.math.columbia.edu/$\sim$szhang/papers/GZSC-2011-8.pdf}.

\end{thebibliography}
\end{document}